\documentclass[10pt,a4paper]{article}

\usepackage{latexsym,amsfonts,amsmath,amssymb,mathrsfs,url,color,amsthm}
\usepackage{graphicx,cite}
\usepackage{siunitx}
\usepackage{color}

\usepackage{natbib}
 \bibpunct[, ]{(}{)}{,}{a}{}{,}%
 \def\newblock{\ }%
\newtheorem{theorem}{Theorem}
\newtheorem{lemma}{Lemma}
\newtheorem{algorithm}{Algorithm}
\newtheorem{remark}{Remark}

\newtheorem{definition}{Definition}
\newtheorem{corollary}{Corollary}

\newcommand{\rd}{\, \mathrm{d}}

\newcommand{\bszero}{\boldsymbol{0}}
\newcommand{\bsone}{\boldsymbol{1}}
\newcommand{\bsc}{\boldsymbol{c}}
\newcommand{\bsd}{\boldsymbol{d}}

\newcommand{\bsk}{\boldsymbol{k}}
\newcommand{\bsl}{\boldsymbol{l}}
\newcommand{\bsx}{\boldsymbol{x}}
\newcommand{\bsy}{\boldsymbol{y}}

\newcommand{\bsalpha}{\boldsymbol{\alpha}}
\newcommand{\bsbeta}{\boldsymbol{\beta}}
\newcommand{\bsgamma}{\boldsymbol{\gamma}}

\newcommand{\bstau}{\boldsymbol{\tau}}
\newcommand{\Dcal}{\mathcal{D}}
\newcommand{\Scal}{\mathcal{S}}
\newcommand{\FF}{\mathbb{F}}
\newcommand{\NN}{\mathbb{N}}
\newcommand{\RR}{\mathbb{R}}
\newcommand{\ZZ}{\mathbb{Z}}
\newcommand{\wal}{\mathrm{wal}}
\newcommand{\wor}{\mathrm{wor}}
\newcommand{\tr}{\mathrm{tr}}

\allowdisplaybreaks

\begin{document}

\title{Richardson extrapolation allows truncation of higher order digital nets and sequences\thanks{This work was supported by JSPS Grant-in-Aid for Young Scientists No.~15K20964.}}

\author{Takashi Goda\thanks{School of Engineering, The University of Tokyo, 7-3-1 Hongo, Bunkyo-ku, Tokyo 113-8656, Japan (\tt{goda@frcer.t.u-tokyo.ac.jp})}}

\date{\today}

\maketitle

\begin{abstract}
We study numerical integration of smooth functions defined over the $s$-dimensional unit cube. A recent work by \citet{DGYxx} has introduced so-called extrapolated polynomial lattice rules, which achieve the almost optimal rate of convergence for numerical integration and can be constructed by the fast component-by-component search algorithm with smaller computational costs as compared to interlaced polynomial lattice rules. In this paper we prove that, instead of polynomial lattice point sets, truncated higher order digital nets and sequences can be used within the same algorithmic framework to explicitly construct good quadrature rules achieving the almost optimal rate of convergence. The major advantage of our new approach compared to original higher order digital nets is that we can significantly reduce the precision of points, i.e., the number of digits necessary to describe each quadrature node. This finding has a practically useful implication when either the number of points or the smoothness parameter is so large that original higher order digital nets require more than the available finite-precision floating point representations.
\end{abstract}

\section{Introduction}
In this paper we study numerical integration of multivariate functions defined over the $s$-dimensional unit cube. For an integrable function $f\colon [0,1)^s\to \RR$, we denote the integral of $f$ by
\begin{equation*}
I(f) = \int_{[0,1)^s}f(\bsx)\rd \bsx.
\end{equation*}
We consider approximating $I(f)$ by a linear algorithm of the form
\begin{equation*}
A_{N}(f) = \sum_{h=0}^{N-1}w_hf(\bsx_h), 
\end{equation*}
where $\bsx_0,\ldots,\bsx_{N-1}$ and $w_0,\ldots,w_{N-1}$ denote the quadrature nodes and the quadrature weights, respectively. We call the algorithm $A_{N}$ a \emph{quasi-Monte Carlo (QMC) rule} if the weights are given by $w_0=\cdots =w_{N-1}=1/N$.

For a Banach space $V$ with norm $\| \cdot\|_{V}$, the worst-case error of $A_{N}$ is defined by
\begin{equation*}
e^{\wor}(A_{N}, V) := \sup_{\substack{f\in V\\ \|f\|_{V}\leq 1}}\left| I(f)-A_{N}(f)\right|.
\end{equation*}
Our aim is then to design a good quadrature rule $A_{N}$ such that $e^{\wor}(A_{N}, V)$ is made as small as possible, since for any function $f\in V$ we have
\begin{equation*}
\left| I(f)-A_{N}(f)\right| \leq \| f\|_{V}\cdot e^{\wor}(A_{N}, V),
\end{equation*}
meaning that a single algorithm works well for all functions belonging to $V$. In this paper we are particularly interested in Banach spaces with dominating mixed smoothness $\alpha\in \NN$, $\alpha\geq 2$, consisting of functions which have partial mixed derivatives up to order $\alpha$ in each variable (see Section~\ref{subsec:sobolev} for more details). Such function spaces have been motivated by \citet{DKLNS14} for the study of partial differential equations with random coefficients. 

For function spaces of our interest, QMC rules using \emph{higher order digital nets and sequences} as quadrature nodes are known to achieve the almost optimal rate of convergence of the worst-case error, which is $O(N^{-\alpha+\varepsilon})$ with arbitrarily small $\varepsilon>0$. The concept and explicit construction of higher order digital nets and sequences were originally introduced by \citet{D07,D08} (see Section~\ref{subsec:ho_nets} for more details). Since then, on the one hand, further theoretical investigations on them have been made \citep[see, e.g.,][]{BDP11,HMOT16,GSY17,GSY18}. On the other hand, how to efficiently search for good quadrature node sets in a weighted function space setting as considered by \citet{SW98} has also attracted some interest \citep[see, e.g.,][]{BDLNP12,GD15,G15,GS16,GSY16}. In particular, so-called interlaced polynomial lattice rules originated by \citet{GD15} and \citet{G15}, which are based on the digit interlacing composition due to \citet{D07,D08}, have been applied in the context of partial differential equations with random coefficients \citep[see, e.g.,][]{DKLNS14,KN16}.

Recently, a new alternative approach to interlaced polynomial lattice rules has been developed by \citet{DGYxx}. Instead of searching for a single interlaced polynomial lattice point set, their approach is to search for $\alpha$ classical polynomial lattice point sets with geometric spacing of $N$ first, and then to apply Richardson extrapolation recursively to $\alpha$ numerical values $A_{N}(f)$. Such \emph{extrapolated polynomial lattice rules} have been proved to achieve the almost optimal rate of convergence, and, moreover, the fast component-by-component algorithm can be used to find good rules with smaller computational costs as compared to interlaced polynomial lattice rules. A further advantage can be found in the fact that the fast QMC matrix-vector multiplication technique from \citet{DKLS15} applies to extrapolated polynomial lattice rules, whereas it is not straightforwardly applicable to interlaced ones.

In this paper, as a continuation of \citet{DGYxx}, we push forward the idea of applying Richardson extrapolation to QMC rules for achieving a high order of convergence for multivariate numerical integration. In particular, we consider QMC rules using \emph{truncated} higher order digital nets or sequences as quadrature nodes, where truncation is done in the following way: we apply the following map $\tr_m \colon [0,1)\to [0,1)$ component-wise to each node $\bsx_h=(x_{h,1}.\ldots,x_{h,s})\in [0,1)^s$ of higher order digital nets with prime base $p$ and size $N=p^m$:
\begin{equation}\label{eq:trunction}
 \tr_m\left(\sum_{i=1}^{\infty}\frac{\xi_i}{p^i} \right)= \sum_{i=1}^{m}\frac{\xi_i}{p^i} \quad \text{with $\xi_i\in \{0,1,\ldots,p-1\}$.} 
\end{equation}
Then we prove that, by applying Richardson extrapolation recursively to QMC rules using such truncated higher order digital nets or sequences with geometric spacing of $N$, the resulting linear algorithm to approximate $I(f)$ achieves the almost optimal rate of convergence.

Our finding has the following practically useful implication, especially when $p=2$. The original digit interlacing composition approach to constructing higher order digital nets with size $N=p^m$ requires $\alpha m$ digits in the $p$-adic expansion of each component of each node. Hence the round-off error is inevitable when $\alpha m$ is larger than what is available via finite-precision floating point representations (for instance, 23 and 52 for IEEE 754 single- and double-precision floating-point formats, respectively). Depending on an integrand, the round-off error becomes comparable to the approximation error for numerical integration already when $m$ is of practical size, say $m\approx 20$. In such a situation, the approximation error will remain more or less unchanged even by increasing $m$. Since our extrapolation approach can reduce the necessary number of digits from $\alpha m$ to $m$, the round-off error problem will not happen until $m$ is large enough and importantly becomes independent of the smoothness parameter $\alpha$. Therefore, with the help of Richardson extrapolation, higher order QMC rules become available for wider ranges of $N$ and $\alpha$ than before without suffering from the rounding problem.

The rest of this paper is organized as follows. After describing the necessary background and notation in Section~\ref{sec:pre}, we propose an extrapolation-based quadrature rule using truncated higher order digital nets or sequences, and prove the worst-case error bound of the proposed rule in Banach spaces with dominating mixed smoothness in Section~\ref{sec:main}. In the same section, we further provide another possible, similar but different quadrature rule, together with its worst-case error bound. We conclude this paper with numerical experiments in Section~\ref{sec:exp}.

\section{Preliminaries}\label{sec:pre}
Throughout this paper we denote the set of positive integers by $\NN$ and write $\NN_0=\NN\cup \{0\}$. For a prime $p$, let $\FF_p$ be the finite field with $p$ elements, which is identified with the set of integers $\{0,1,\ldots, p-1\}\subset \ZZ$ equipped with addition and multiplication modulo $p$. For an $s$-dimensional vector $\bsx=(x_1,\ldots,x_s)$ and a subset $u\subseteq \{1,\ldots,s\}$, we write $\bsx_u=(x_j)_{j\in u}$, and denote the cardinality and the complement of $u$ by $|u|$ and $-u:=\{1,\ldots,s\}\setminus u$, respectively.

\subsection{Banach spaces with dominating mixed smoothness}\label{subsec:sobolev}
Following \citet{DKLNS14}, here we introduce the definition of function spaces which we consider in this paper.  Let $\alpha\in \NN$, $\alpha\geq 2$, and $1\leq q,r \leq \infty$ be real numbers. Further let $\bsgamma=(\gamma_u)_{u\subseteq \{1,\ldots,s\}}$ be a set of non-negative real numbers called weights, which has been introduced by \citet{SW98} to moderate the relative importance of different variables or groups of variables. In this paper we do not discuss the dependence of the worst-case error on the dimension, and just consider the weights for making consistent use of the notations of previous works. 

Assume that a function $f\colon [0,1)^s\to \RR$ has partial mixed derivatives up to order $\alpha$ in each variable. We define the norm of $f$ by
\begin{equation*}
\|f\|^r_{s,\alpha,q, r} := \sum_{u\subseteq \{1,\ldots,s\}}\left( \gamma_u^{-q}\sum_{v\subseteq u}\sum_{\bstau_{u\setminus v}\in \{1,\ldots,\alpha\}^{|u\setminus v|}} \int_{[0,1)^{|v|}}\left| \int_{[0,1)^{s-|v|}}f^{(\bstau_{u\setminus v},\bsalpha_{v},\bszero)}(\bsx)\rd \bsx_{-v}\right|^q \rd \bsx_v\right)^{r/q},
\end{equation*}
with the obvious modification if either $q$ or $r$ is infinite. Here $(\bstau_{u\setminus v},\bsalpha_{v},\bszero)$ denotes the vector $\bsbeta=(\beta_1,\ldots,\beta_s)$ such that
\begin{equation*} \beta_j= \begin{cases} 
\tau_j & \text{if $j\in u\setminus v$,} \\
\alpha & \text{if $j\in v$,} \\
0 & \text{otherwise,}
\end{cases}
\end{equation*}
and $f^{(\bstau_{u\setminus v},\bsalpha_{v},\bszero)}$ denotes the partial mixed derivative of order $(\bstau_{u\setminus v},\bsalpha_{v},\bszero)$ of $f$. If there exist subsets $u$ such that $\gamma_u=0$, then we assume that the corresponding inner double sum is 0 and formally set $0/0=0$. Now we define the Banach space with dominating mixed smoothness $\alpha$ by
\begin{equation*}
W_{s,\alpha,q,r} := \left\{ f\colon [0,1)^s\to \RR : \|f\|_{s,\alpha,q, r}<\infty \right\}.
\end{equation*}

For $\tau\in \NN$, we denote the Bernoulli polynomial of degree $\tau$ by $B_{\tau}\colon [0,1)\to \RR$ and we put $b_{\tau}(\cdot)=B_{\tau}(\cdot)/\tau!$. With a slight abuse of notation, we write $b_{\tau}=b_{\tau}(0)$. Further, we denote the one-periodic extension of the polynomial $b_{\tau}$ by $\tilde{b}_{\tau}\colon \RR\to \RR$. As shown below, we have a point-wise representation for functions in $W_{s,\alpha,q,r}$.
\begin{lemma}\label{lem:func_decomp}
For $f\in W_{s,\alpha,q,r}$, we have
\begin{equation*}
f(\bsx) = \sum_{u\subseteq \{1,\ldots,s\}}f_u(\bsx_u),
\end{equation*}
where each $f_u$ depends only on $\bsx_u$ and is given by
\begin{equation*}
f_u(\bsx_u) = \sum_{v\subseteq u} (-1)^{(\alpha+1)|v|}\sum_{\bstau_{u\setminus v}\in \{1,\ldots,\alpha\}^{|u\setminus v|}}\prod_{j\in u\setminus v}b_{\tau_j}(x_j)\int_{[0,1)^s}f^{(\bstau_{u\setminus v}, \bsalpha_v,\bszero)}(\bsy)\prod_{j\in v}\tilde{b}_{\alpha}(y_j-x_j)\rd \bsy.
\end{equation*}
Moreover, we have
\begin{equation*}
\| f\|^r_{s,\alpha,q,r} = \sum_{u\subseteq \{1,\ldots,s\}}\|f_u\|^r_{s,\alpha,q,r}.
\end{equation*}
\end{lemma}
\begin{proof}
See the proof of \citet[Theorem~3.5]{DKLNS14}.
\end{proof}

\subsection{Higher order digital nets and sequences}\label{subsec:ho_nets}

\subsubsection{Digital construction scheme}
We first introduce a class of point sets called digital nets which are originally due to \citet{Nbook}.
\begin{definition}[Digital nets]\label{def:digital_net}
For a prime $p$ and $m,n\in \NN$, let $C_1,\ldots,C_s\in \FF_p^{n\times m}$. For $h\in \NN_0$, $h<p^m$, we denote the $p$-adic expansion of $h$ by
\begin{equation*}
h = \eta_0 + \eta_1 p+\cdots + \eta_{m-1}p^{m-1}.
\end{equation*}
Set $\bsx_h=(x_{h,1},\ldots,x_{h,s})\in [0,1)^s$ where
\begin{equation*}
x_{h,j}=\frac{\xi_{h,j,1}}{p}+\frac{\xi_{h,j,2}}{p^2}+\cdots+\frac{\xi_{h,j,n}}{p^n},
\end{equation*}
in which $\xi_{h,j,1},\ldots,\xi_{h,j,n}$ are given by
\begin{equation*}
(\xi_{h,j,1},\ldots,\xi_{h,j,n}) = (\eta_0,\eta_1,\ldots,\eta_{m-1})\cdot C_j^{\top}
\end{equation*}
for $1\leq j\leq s$. Then the set of points $P_{m,n}=\{\bsx_h : 0\leq h<p^m\}$ is called a digital net over $\FF_p$ (with generating matrices $C_1,\ldots,C_s$).
\end{definition}
\noindent It is easy to see from the definition that the parameter $m$ determines the total number of points, while $n$ does determine the precision of points.

\begin{remark}\label{rem:finiteness}
Let us consider the case $n=\infty$. For each $C_j=(c_{k,l}^{(j)})_{k\in \NN, 1\leq l\leq m}$, if there exists a function $K_j\colon \{1,\ldots,m\}\to \NN$ such that $c_{k,l}^{(j)}=0$ whenever $k>K_j(l)$, the vector-matrix product appearing in the above definition gives $\xi_{h,j,i}=0$ for all $i>\max_{1\leq l\leq m}K_j(l)=:n_j$. Thus, each number $x_{h,j}$ is uniquely written in a finite $p$-adic expansion with the precision at most $n'=\max_{1\leq j\leq s}n_j$. By identifying $C_1,\ldots,C_s$ with their upper $n'\times m$ submatrices, Definition~\ref{def:digital_net} still applies to such cases.
\end{remark}

It is straightforward to extend the definition of digital nets to digital sequences that are infinite sequences of points in $[0,1)^s$.
\begin{definition}[Digital sequences]\label{def:digital_seq}
For a prime $p$, let $C_1,\ldots,C_s\in \FF_p^{\NN\times \NN}$. For each $C_j=(c_{k,l}^{(j)})_{k,l\in \NN}$, assume that there exists a function $K_j\colon \NN\to \NN$ such that $c_{k,l}^{(j)}=0$ if $k>K_j(l)$. For $h\in \NN_0$, we denote the $p$-adic expansion of $h$ by
\begin{equation*}
h = \eta_0 + \eta_1 p+\cdots ,
\end{equation*}
where all but a finite number of $\eta_i$'s are 0.
Set $\bsx_h=(x_{h,1},\ldots,x_{h,s})\in [0,1)^s$ where
\begin{equation*}
x_{h,j}=\frac{\xi_{h,j,1}}{p}+\frac{\xi_{h,j,2}}{p^2}+\cdots,
\end{equation*}
in which $\xi_{h,j,1},\xi_{h,j,2},\ldots$ are given by
\begin{equation*}
(\xi_{h,j,1},\xi_{h,j,2},\ldots) = (\eta_0,\eta_1,\ldots)\cdot C_j^{\top}
\end{equation*}
for $1\leq j\leq s$. Then the sequence of points $\Scal=\{\bsx_h : h\in \NN_0\}$ is called a digital sequence over $\FF_p$ (with generating matrices $C_1,\ldots,C_s$).
\end{definition}
\noindent As mentioned in Remark~\ref{rem:finiteness}, the existence of functions $K_j$ in this definition is assumed to ensure that every number $x_{h,j}$ is uniquely written in a finite $p$-adic expansion.

\subsubsection{Dual nets}
Next we introduce the concept of dual nets and also the weight function due to \citet{D08} which generalizes the original weight function introduced independently by \citet{N86} and \citet{RT97}. Thereafter we give the definition of higher order digital nets and sequences.
\begin{definition}[Dual nets]\label{def:dual_net}
For a prime $p$ and $m,n\in \NN$, let $P_{m,n}$ be a digital net over $\FF_p$ with generating matrices $C_1,\ldots,C_s\in \FF_p^{n\times m}$. The dual net of $P_{m,n}$, denoted by $P_{m,n}^{\perp}$, is defined by
\begin{equation*}
P_{m,n}^{\perp}:= \left\{ \bsk=(k_1,\ldots,k_s)\in \NN_0^s : C_1^{\top}\nu_n(k_1)\oplus \cdots \oplus C_s^{\top}\nu_n(k_s)=\bszero \in \FF_p^m\right\},
\end{equation*}
where we write $\nu_n(k)=(\kappa_0,\ldots,\kappa_{n-1})^{\top}$ for $k\in \NN_0$ whose $p$-adic expansion is given by $k= \kappa_0 + \kappa_1 p + \cdots$, where all but a finite number of $\kappa_i$'s are 0.
\end{definition}

\begin{remark}\label{rem:finiteness2}
Again, even for the case $n=\infty$, as long as there exists a function $K_j\colon \{1,\ldots,m\}\to \NN$ such that $c_{k,l}^{(j)}=0$ whenever $k>K_j(l)$ for each $C_j=(c_{k,l}^{(j)})_{k,l\in \NN}$, Definition~\ref{def:dual_net} still applies.
\end{remark}

\begin{definition}[Weight function]
Let $\alpha\in \NN$. We denote the $p$-adic expansion of $k\in \NN$ by
\begin{equation*}
k= \kappa_1 p^{c_1-1}+\kappa_2p^{c_2-1}+\cdots + \kappa_v p^{c_v-1}
\end{equation*}
with $\kappa_1,\ldots,\kappa_v\in \{1,\ldots,p-1\}$ and $c_1>c_2>\cdots >c_v>0$. Then we define the weight function $\mu_{\alpha}\colon \NN_0\to \NN_0$ by
\begin{equation*}
\mu_{\alpha}(k):=\sum_{i=1}^{\min(\alpha,v)}c_i,
\end{equation*}
and $\mu_{\alpha}(0)=0$. In case of vectors in $\NN_0^s$, we define
\begin{equation*}
\mu_{\alpha}(k_1,\ldots,k_s):=\sum_{j=1}^{s}\mu_{\alpha}(k_j). 
\end{equation*}
\end{definition}

Now we are ready to introduce higher order digital nets and sequences.
\begin{definition}[Higher order digital nets]\label{def:HO_digital_net}
Let $\alpha\in \NN$. For a prime $p$ and $m,n\in \NN$, let $P_{m,n}$ be a digital net over $\FF_p$. We call $P_{m,n}$ an \emph{order $\alpha$ digital $(t,m,s)$-net} over $\FF_p$ if there exists an integer $0\leq t\leq \alpha m$ such that the following holds:
\begin{equation*}
\mu_{\alpha}(P_{m,n}^{\perp}):= \min_{\bsk\in P_{m,n}^{\perp}\setminus \{\bszero\}}\mu_{\alpha}(\bsk)>\alpha m-t. 
\end{equation*}
\end{definition}
\begin{remark}\label{rem:precision}
It follows from Definition~\ref{def:dual_net} that $(p^n,0,\ldots,0)\in P_{m,n}^{\perp}$, which gives
\begin{equation*}
\mu_{\alpha}(P_{m,n}^{\perp})\leq \mu_{\alpha}(p^n,0,\ldots,0)=n+1.
\end{equation*}
Thus in order for $P_{m,n}$ to be an order $\alpha$ digital $(t,m,s)$-net, it is necessary to have $n\geq \alpha m-t$. Together with Remark~\ref{rem:digit_interlacing} below, this means that the precision $n$ should scale linearly with $\alpha$ and $m$.
\end{remark}
\begin{definition}[Higher order digital sequences]\label{def:HO_digital_seq}
Let $\alpha\in \NN$. For a prime $p$, let $\Scal$ be a digital sequence over $\FF_p$. We call $\Scal$ an \emph{order $\alpha$ digital $(t,s)$-sequence} over $\FF_p$ if there exists $t\in \NN_0$ such that the first $p^m$ points of $\Scal$ are an order $\alpha$ digital $(t,m,s)$-net over $\FF_p$ when $\alpha m>t$.
\end{definition}

\subsubsection{Explicit constructions}
It is important to note that higher order digital nets and sequences can be constructed explicitly. In fact, many explicit constructions of order 1 digital $(t,s)$-sequences with small $t$-values for arbitrary $s$ have been known already. Among them are those by \citet{S67}, \citet{F82}, \citet{N88}, \citet{T93} and \citet{NXbook}. Some of them hold the property on functions $K_j$ in Definition~\ref{def:digital_seq} such that $K_j(l)\leq l$ for all $j,l\in \NN$. This means, the first $p^m$ points of such digital sequences are an order 1 digital $(t,m,s)$-net over $\FF_p$ with the precision $n\leq m$. We refer to \citet[Chapter~8]{DPbook} for more information on these special constructions. 

Moreover the digit interlacing composition due to \citet{D07,D08} enables us to construct order $\alpha$ digital $(t,m,s)$-nets and $(t,s)$-sequences in the following way. For $\alpha\in \NN$, $\alpha\geq 2$, let us consider a generic point $\bsx=(x_1,\ldots,x_{\alpha})\in [0,1)^{\alpha}$. We denote the $p$-adic expansion of each $x_j$ by
\begin{equation*}
x_j = \frac{\xi_{j,1}}{p}+\frac{\xi_{j,2}}{p^2}+\cdots, 
\end{equation*}
which is understood to be unique in the sense that infinitely many of the $\xi_{j,i}$'s are different from $p-1$. Then we define the map $\Dcal_{\alpha}\colon [0,1)^{\alpha}\to [0,1)$ by
\begin{equation*}
\Dcal_{\alpha}(x_1,\ldots,x_{\alpha}) := \sum_{i=1}^{\infty}\sum_{j=1}^{\alpha}\frac{\xi_{j,i}}{p^{\alpha(i-1)+j}}.
\end{equation*}
We extend the map $\Dcal_{\alpha}$ to vectors by setting
\begin{align*}
\Dcal_{\alpha} \colon [0,1)^{\alpha s} & \to [0,1)^s, \\
(x_1,\ldots,x_{\alpha s}) & \mapsto \left( \Dcal_{\alpha}(x_1,\ldots,x_\alpha),\ldots,\Dcal_{\alpha}(x_{\alpha(s-1)+1},\ldots,x_{\alpha s})\right),
\end{align*}
i.e., $\Dcal_{\alpha}$ is applied to non-overlapping consecutive $\alpha$ components of $(x_1,\ldots,x_{\alpha s})$. Using this digit interlacing composition $\Dcal_{\alpha}$, we can construct higher order digital nets and sequences explicitly as follows.

\begin{lemma}
Let $\alpha\in \NN$, $\alpha\geq 2$, and $p$ be a prime.
\begin{enumerate}
\item For $m\in \NN$, let $P_{m,m}$ be an order 1 digital $(t,m,\alpha s)$-net over $\FF_p$. Then
\begin{equation*}
\Dcal_{\alpha}(P_{m,m}) := \left\{ \Dcal_{\alpha}(\bsx) : \bsx\in P_{m,m}\right\} \subset [0,1)^s
\end{equation*}
is an order $\alpha$ digital $(t',m,s)$-net over $\FF_p$ with
\begin{equation*}
t' = \alpha \min\left\{ m, t+\left\lfloor \frac{s(\alpha-1)}{2}\right\rfloor\right\}.
\end{equation*}
\item Let $\Scal$ be an order 1 digital $(t,\alpha s)$-sequence over $\FF_p$. Then
\begin{equation*}
\Dcal_{\alpha}(\Scal) := \left\{ \Dcal_{\alpha}(\bsx) : \bsx\in \Scal\right\} \subset [0,1)^s
\end{equation*}
is an order $\alpha$ digital $(t',s)$-sequence over $\FF_p$ with
\begin{equation*}
t' = \alpha t+\frac{s\alpha (\alpha-1)}{2}.
\end{equation*}
\end{enumerate}
\end{lemma}
\begin{proof}
See \citet[Corollary~3.4]{BDP11} and \citet[Theorems~4.11 and 4.12]{D08} for the proofs of the first and second items, respectively. 
\end{proof}

\begin{remark}\label{rem:digit_interlacing}
Let $\Scal$ be an order 1 digital $(t,\alpha s)$-sequence over $\FF_p$ with generating matrices $C_1,\ldots,C_{\alpha s}\in \FF_p^{\NN\times \NN}$. We denote the $l$-th row of $C_j$ by $\bsc_{l}^{(j)}$. Then $\Dcal_{\alpha}(\Scal)$ is a digital sequence over $\FF_p$ with generating matrices $D_1,\ldots,D_s\in \FF_p^{\NN\times \NN}$, where each $D_j$ whose $l$-th row is denoted by $\bsd_l^{(j)}$ is given by
\begin{equation*}
\bsd_{\alpha(l-1)+h}^{(j)} = \bsc_{l}^{(\alpha(j-1)+h)}
\end{equation*}
for $l\geq 1$ and $1\leq h\leq \alpha$. If each $C_j=(c_{k,l}^{(j)})$ satisfies $c_{k,l}^{(j)}=0$ whenever $k>l$, i.e., if $K_j(l)\leq l$ holds, each $D_j=(d_{k,l}^{(j)})$ satisfies $d_{k,l}^{(j)}=0$ whenever $k>\alpha l$. This means, the first $p^m$ points of $\Dcal_{\alpha}(\Scal)$ are an order $\alpha$ digital $(t',m,s)$-net over $\FF_p$ with the precision $n\leq \alpha m$.

Since several explicit constructions of order 1 digital $(t,\alpha s)$-sequences, including those of \citet{S67} and \citet{T93}, fulfill the condition $K_j(l)\leq l$, we assume that the precision of the first $p^m$ points of an order $\alpha$ digital $(t,s)$-sequence over $\FF_p$ is at most $\alpha m$ in the rest of this paper.
\end{remark}

\subsection{Walsh functions}
Finally, in this section, we recall the definition of Walsh functions which play a central role in the quadrature error analysis of QMC rules using (higher order) digital nets and sequences.

\begin{definition}[Walsh functions]
For a prime $p$, we write $\omega_p=\exp(2\pi \sqrt{-1}/p)$. For $k\in \NN_0$ whose $p$-adic expansion is given by $k=\kappa_0+\kappa_1p+\cdots$, where all but a finite number of $\kappa_i$'s are 0, the $k$-th Walsh function $\wal_k\colon [0,1)\to \{1, \omega_p,\ldots, \omega_p^{p-1} \}$ is defined by
\begin{equation*}
\wal_k(x) := \omega_p^{\kappa_0\xi_1+\kappa_1\xi_2+\cdots},
\end{equation*}
where we denote the $p$-adic expansion of $x\in [0,1)$ by $x=\xi_1/p+\xi_2/p^2+\cdots$, which is understood to be unique in the sense that infinitely many of the $\xi_i$'s are different from $p-1$.

In the multivariate case, for $\bsk=(k_1,\ldots,k_s)\in \NN_0^s$ and $\bsx=(x_1,\ldots,x_s)\in [0,1)^s$, the $\bsk$-th Walsh function is defined by
\begin{equation*}
\wal_{\bsk}(\bsx) :=\prod_{j=1}^{s}\wal_{k_j}(x_j).
\end{equation*}
\end{definition}

\begin{lemma}\label{lem:walsh_grid}
For $k\in \NN_0$ and $n\in \NN$ we have
\begin{equation*}
\frac{1}{p^n}\sum_{h=0}^{p^n-1}\wal_k\left(\frac{h}{p^n}\right)=\begin{cases}
1 & \text{if $p^n$ divides $k$,} \\
0 & \text{otherwise.}
\end{cases}
\end{equation*}
\end{lemma}
\begin{proof}
Write $k=k'+p^n l$ for $0\leq k'<p^n$ and $l\geq 0$. From the definition of Walsh functions, we see that $\wal_k(h/p^n)=\wal_{k'}(h/p^n)$ for any $0\leq h<p^n$. Thus it suffices to prove the result for the case $0\leq k<p^n$. Actually, the result for $k=0$ is trivial and the proof for $1\leq k<p^n$ can be found in \citet[Lemma~A.8]{DPbook}.
\end{proof}

\begin{lemma}\label{lem:walsh_dual}
For a prime $p$ and $m,n\in \NN$, let $P_{m,n}=\{\bsx_h : 0\leq h<p^m\}$ be a digital net over $\FF_p$.  For $\bsk\in \NN_0^s$ we have
\begin{equation*}
\frac{1}{p^m}\sum_{h=0}^{p^m-1}\wal_{\bsk}(\bsx_h)=\begin{cases}
1 & \text{if $\bsk\in P_{m,n}^{\perp}$,} \\
0 & \text{otherwise.}
\end{cases}
\end{equation*}
\end{lemma}
\begin{proof}
See \citet[Lemma~4.75]{DPbook} for the proof.
\end{proof}

As shown in \citet[Theorem~A.11]{DPbook}, the system $\{\wal_{\bsk} : \bsk\in \NN_0^s\}$ is a complete orthonormal system in $L_2([0,1)^s)$. Therefore, we can define the Walsh series of $f\in L_2([0,1)^s)$
\begin{equation*}
\sum_{\bsk\in \NN_0^s}\hat{f}(\bsk)\wal_{\bsk}(\bsx),
\end{equation*}
where $\hat{f}(\bsk)$ is the $\bsk$-th Walsh coefficient of $f$:
\begin{equation*}
\hat{f}(\bsk):=\int_{[0,1)^s}f(\bsx)\overline{\wal_{\bsk}(\bsx)}\rd \bsx.
\end{equation*}
It is easy to see that $I(f)=\hat{f}(\bszero)$.

For smooth functions $f\in W_{s,\alpha,q,r}$, the above Walsh series converges to $f$ point-wise absolutely, and moreover, the following bounds on the Walsh coefficients are known.
\begin{lemma}\label{lem:walsh_decay}
Let $u$ be a subset of $\{1,\ldots,s\}$ and $\bsk_u\in \NN^{|u|}$. The $(\bsk_u,\bszero)$-th Walsh coefficient of $f\in W_{s,\alpha,q,r}$ is bounded by
\begin{equation*}
|\hat{f}(\bsk_u,\bszero)| \leq \gamma_u \| f_u\|_{s,\alpha,q,r} C_{\alpha}^{|u|}p^{-\mu_{\alpha}(\bsk_u)}, 
\end{equation*}
where $f_u$ is given as in Lemma~\ref{lem:func_decomp} and
\begin{align*}
C_{\alpha}  = \left( 1+\frac{1}{p}+\frac{1}{p(p+1)}\right)^{\alpha-2} \left( 3+\frac{2}{p}+\frac{2p+1}{p-1}\right)\max\left( \frac{2}{(2\sin (\pi/p))^{\alpha}}, \max_{1\leq z<\alpha} \frac{1}{(2\sin(\pi/p))^z}\right).
\end{align*}
\end{lemma}
\begin{proof}
See \citet[Theorem~15]{D09} and \citet[Theorem~3.5]{DKLNS14} for the proof.
\end{proof}

\section{Extrapolation of truncated higher order digital nets and sequences}\label{sec:main}

\subsection{Euler-Maclaurin formula}
Before providing our extrapolation-based quadrature rules, here we show some necessary results as preparation. In what follows, for $l\in \NN$ and $\bsk=(k_1,\ldots,k_s)\in \NN_0^s$, we write $l\mid \bsk$ if $l$ divides $k_j$ for all $1\leq j\leq s$, and $l\nmid \bsk$ if there exists at least one component $k_j$ which is not divided by $l$. Further we write $\bsk<l$ (resp. $\bsk>l$) if $k_j<l$ (resp. $k_j>l$) holds for all $1\leq j\leq s$.

\begin{lemma}\label{lem:dual_digital_net}
For a prime $p$ and $m,n\in \NN$, let $P_{m,n}$ be a digital net over $\FF_p$ with generating matrices $C_1,\ldots,C_s\in \FF_p^{n\times m}$. Then we have
\begin{align*}
\left\{\bsk\in \NN_0^s : p^n\mid \bsk\right\} \subseteq P_{m,n}^{\perp} ,
\end{align*}
and
\begin{align*}
P_{m,n}^{\perp}\setminus \left\{\bsk\in \NN_0^s : p^n\mid \bsk\right\} = \left\{ \bsk+p^n\bsl : \bsk\in P_{m,n}^{\perp}, \bszero\neq \bsk<p^n, \bsl\in \NN_0^s \right\} .
\end{align*}
\end{lemma}
\begin{proof}
The first statement is trivial, since for $\bsk\in \NN_0^s$ such that $p^n\mid \bsk$, we have $\nu_n(k_1)=\cdots=\nu_n(k_s)=(0,\ldots,0)^{\top}$ which gives
\begin{align*}
C_1^{\top}\nu_n(k_1)\oplus \cdots \oplus C_s^{\top}\nu_n(k_s)=\bszero. 
\end{align*}
Hence $P_{m,n}^{\perp}$ always contains such $\bsk$ as elements.

Let us move on to the proof of the second statement. For $k,l\in \NN_0$ with $k<p^n$, we have $\nu_n(k+p^n l)=\nu_n(k)$. This means that, for $\bsk,\bsl\in \NN_0^s$ with $\bsk<p^n$, whether $P_{m,n}^{\perp}$ contains $\bsk+p^n \bsl$ as an element does not depend on $\bsl$, so that $\bsk+p^n \bsl\in P_{m,n}^{\perp}$ if and only if $\bsk\in P_{m,n}^{\perp}$. Therefore we have
\begin{align*}
P_{m,n}^{\perp} & = \{\bsk+p^n\bsl : \bsk\in P_{m,n}^{\perp}, \bsk<p^n, \bsl\in \NN_0^s\} \\
& = \{\bsk+p^n\bsl : \bsk\in P_{m,n}^{\perp}, \bszero\neq \bsk<p^n, \bsl\in \NN_0^s\} \cup \{p^n\bsl : \bsl\in \NN_0^s \} ,
\end{align*}
where the last equality follows by separating the cases $\bsk\neq \bszero$ and $\bsk=\bszero$. Since the two sets on the right-most side above are disjoint, the result follows.
\end{proof}

\begin{corollary}\label{cor:euler-mac}
For a prime $p$ and $m,n\in \NN$, let $P_{m,n}=\{\bsx_h : 0\leq h<p^m\}$ be a digital net over $\FF_p$ with generating matrices $C_1,\ldots,C_s\in \FF_p^{n\times m}$. For $f\in W_{s,\alpha,q,r}$ we have
\begin{equation}\label{eq:euler-mac}
 \frac{1}{p^m}\sum_{h=0}^{p^m-1}f(\bsx_h) = I(f)+\sum_{\bsl\in \NN_0^s}\sum_{\substack{\bsk\in P_{m,n}^{\perp}\\ \bszero\neq \bsk<p^n}}\hat{f}(\bsk+p^n\bsl)+\sum_{\tau=1}^{\alpha-1}\frac{c_\tau(f)}{p^{\tau n}}+R_{s,\alpha,n} ,
\end{equation}
where $c_{\tau}(f)$ depends only on $f$ and $\tau$, and $R_{s,\alpha,n}\in O(p^{-\alpha n})$.
\end{corollary}
\begin{proof}
Using the Walsh series of $f$, Lemma~\ref{lem:walsh_dual} and Lemma~\ref{lem:dual_digital_net}, we have
\begin{align}
\frac{1}{p^m}\sum_{h=0}^{p^m-1}f(\bsx_h) & = \frac{1}{p^m}\sum_{h=0}^{p^m-1}\sum_{\bsk\in \NN_0^s}\hat{f}(\bsk)\wal_{\bsk}(\bsx_h) = \sum_{\bsk\in \NN_0^s}\hat{f}(\bsk)\frac{1}{p^m}\sum_{h=0}^{p^m-1}\wal_{\bsk}(\bsx_h) \nonumber \\
& = \sum_{\bsk\in P_{m,n}^{\perp}}\hat{f}(\bsk) = I(f)+\sum_{\bsk\in P_{m,n}^{\perp}\setminus \{\bszero\}}\hat{f}(\bsk) \nonumber \\
& = I(f)+\sum_{\bsl\in \NN_0^s}\sum_{\substack{\bsk\in P_{m,n}^{\perp}\\ \bszero\neq \bsk<p^n}}\hat{f}(\bsk+p^n\bsl)+\sum_{\substack{\bsk\in \NN_0^s\setminus \{\bszero\}\\ p^n\mid \bsk}}\hat{f}(\bsk).\label{eq:expansion_proof_1}
\end{align}

Now we write
\begin{align*}
P_n^* :=\left\{ \left( \frac{h_1}{p^n},\ldots,\frac{h_s}{p^n}\right) : 0\leq h_1,\ldots,h_s<p^n\right\}, 
\end{align*}
which is called a \emph{regular grid}. Using Lemma~\ref{lem:walsh_grid}, for $\bsk\in \NN_0^s$ we have
\begin{equation*}
 \frac{1}{p^{ns}}\sum_{h_1,\ldots,h_s=0}^{p^n-1}\wal_{\bsk}\left( \frac{h_1}{p^n},\ldots,\frac{h_s}{p^n}\right) = \prod_{j=1}^{s}\frac{1}{p^n}\sum_{h_j=0}^{p^n-1}\wal_{k_j}\left( \frac{h_j}{p^n}\right) = \begin{cases}
1 & \text{if $p^n\mid \bsk$,} \\
0 & \text{otherwise.}
\end{cases}
\end{equation*}
Using this result, we obtain
\begin{align*}
\sum_{\substack{\bsk\in \NN_0^s\setminus \{\bszero\}\\ p^n\mid \bsk}}\hat{f}(\bsk) & = \sum_{\bsk\in \NN_0^s\setminus \{\bszero\}}\hat{f}(\bsk)\frac{1}{p^{ns}}\sum_{h_1,\ldots,h_s=0}^{p^n-1}\wal_{\bsk}\left( \frac{h_1}{p^n},\ldots,\frac{h_s}{p^n}\right) \\
& = \frac{1}{p^{ns}}\sum_{h_1,\ldots,h_s=0}^{p^n-1}\sum_{\bsk\in \NN_0^s\setminus \{\bszero\}}\hat{f}(\bsk)\wal_{\bsk}\left( \frac{h_1}{p^n},\ldots,\frac{h_s}{p^n}\right) \\
& = \frac{1}{p^{ns}}\sum_{h_1,\ldots,h_s=0}^{p^n-1}f\left( \frac{h_1}{p^n},\ldots,\frac{h_s}{p^n}\right) -I(f).
\end{align*}
This means that the last term of \eqref{eq:expansion_proof_1} is nothing but a signed integration error of a QMC rule using a regular grid $P_n^*$ as quadrature nodes. It is shown by \citet[Theorem~3.4]{DGYxx} that
\begin{equation}
\frac{1}{p^{ns}}\sum_{h_1,\ldots,h_s=0}^{p^n-1}f\left( \frac{h_1}{p^n},\ldots,\frac{h_s}{p^n}\right) -I(f) = \sum_{\tau=1}^{\alpha-1}\frac{c_\tau(f)}{p^{\tau n}}+R_{s,\alpha,n} ,\label{eq:expansion_proof_2}
\end{equation}
where 
\begin{equation*}
c_\tau(f) = \sum_{\substack{\bstau\in \{0,1,\ldots,\alpha-1\}^s\\ |\bstau|_1=\tau}}I(f^{(\bstau)}) \prod_{\substack{j=1\\ \tau_j\neq 0}}^{s}b_{\tau_j},
\end{equation*}
with $|\bstau|_1=|\tau_1|+\cdots+|\tau_s|$, and
\begin{align}\label{eq:remainder_bound}
 |R_{s,\alpha,n}| & \leq \frac{\|f\|_{s,\alpha,q,r}}{p^{\alpha n}}\left[ \sum_{\emptyset \neq u\subseteq \{1,\ldots,s\}}\left(\gamma_u(\alpha+1)^{|u|/q'}D_{\alpha}^{|u|}\right)^{r'}\right]^{1/r'}  \nonumber \\
 & \leq \frac{\|f\|_{s,\alpha,q,r}}{p^{\alpha n}}\sum_{\emptyset \neq u\subseteq \{1,\ldots,s\}}\gamma_u(\alpha+1)^{|u|}D_{\alpha}^{|u|},
\end{align}
with $q'$ an $r'$ being the H\"{o}lder conjugates of $q$ and $r$, respectively, and 
\begin{equation*}
D_{\alpha}=\max\left\{|b_1|,\ldots,|b_{\alpha-1}|, \sup_{x\in [0,1)}|\tilde{b}_{\alpha}(x)|\right\}. 
\end{equation*}
We complete the proof by substituting \eqref{eq:expansion_proof_2} into the last term of \eqref{eq:expansion_proof_1}.
\end{proof}

\subsection{An algorithm and its worst-case error bound}
Throughout this subsection, let $\Scal$ be an order $\alpha$ digital $(t,s)$-sequence over $\FF_p$ with generating matrices $C_1,\ldots,C_s$. For $m,n\in \NN$, we denote the upper-left $n\times m$ submatrices of $C_1,\ldots,C_s$ by $C_1^{[n\times m]},\ldots,C_s^{[n\times m]}$, and denote a digital net with generating matrices $C_1^{[n\times m]},\ldots,C_s^{[n\times m]}$ by $P^{[n\times m]}$. In view of Remarks~\ref{rem:finiteness} and \ref{rem:digit_interlacing}, we assume that
\begin{equation*}
P^{[\alpha m\times m]} = P^{[(\alpha m+1)\times m]} =\cdots = P^{[\NN\times m]},
\end{equation*}
where the right-most side denotes the first $p^m$ points of $\Scal$. It is easy to see that, for a finite $n$, we have
\begin{equation*}
P^{[n\times m]} = \tr_n(P^{[\NN\times m]}),
\end{equation*}
where the map $\tr_n$ is defined as in \eqref{eq:trunction}, and from Remark~\ref{rem:finiteness2}, we also have
\begin{equation}\label{eq:trunc_dual}
 \left\{ \bsk\in (P^{[n\times m]})^{\perp} : \bsk<p^n\right\} = \left\{ \bsk\in (P^{[\NN\times m]})^{\perp} : \bsk<p^n\right\}. 
\end{equation}
Furthermore, instead of $A_N(f)$, we write
\begin{equation*}
I(f;P) = \frac{1}{N}\sum_{\bsx\in P}f(\bsx)
\end{equation*}
for an $N$-element point set $P\subset [0,1)^s$ to emphasize which point set is used in numerical integration. 

Now let us consider the following algorithm:
\begin{algorithm}\label{alg:first}
Let $\Scal$ be an order $\alpha$ digital $(t,s)$-sequence over $\FF_p$. For $m\in \NN$ and $f\colon [0,1)^s\to \RR$, do the following:
\begin{enumerate}
\item For $0\leq i<\alpha$, compute
\begin{equation*}
I_{m+i}^{(1)}(f) := I\left(f;P^{[(m+i)\times (m+i)]}\right).
\end{equation*}
\item For $1\leq \tau<\alpha$, let
\begin{equation*}
I_{m+i}^{(\tau+1)}(f) := \frac{p^\tau I_{m+i+1}^{(\tau)}(f)-I_{m+i}^{(\tau)}(f)}{p^\tau-1}\quad \text{for $0\leq i<\alpha-\tau$}.
\end{equation*}
\item Return $I_{m}^{(\alpha)}(f)$ as an approximation of $I(f)$.
\end{enumerate}
\end{algorithm}

We emphasize that Algorithm~\ref{alg:first} uses only digital nets with \emph{square} generating matrices, which significantly reduces the necessary precision of points from $\alpha m$ (see Remarks~\ref{rem:precision} and \ref{rem:digit_interlacing}) to $m$. Since the resulting estimate $I_m^{(\alpha)}(f)$ is given by a weighted sum of QMC rules with different sizes of nodes, $I_{m}^{(1)}(f),\ldots,I_{m+\alpha-1}^{(1)}(f)$, this quadrature rule is a linear algorithm with the total number of function evaluations
\begin{equation*}
N=p^{m}+\cdots+p^{m+\alpha-1}.
\end{equation*}
As a main result of this paper, we show that our quadrature rule $I_m^{(\alpha)}(f)$ achieves the almost optimal rate of convergence of the worst-case error in $W_{s,\alpha,q,r}$.
\begin{theorem}\label{thm:main}
Let $\alpha\in \NN$, $\alpha\geq 2$, and $1\leq q,r \leq \infty$. When $\alpha m>t$ holds, the worst-case error of the algorithm $I_m^{(\alpha)}(f)$ in $W_{s,\alpha,q,r}$ is bounded above by
\begin{equation*}
\sup_{\substack{f\in W_{s,\alpha,q,r}\\ \|f\|_{s,\alpha,q,r}\leq 1}}\left| I(f)-I_m^{(\alpha)}(f)\right| \leq \sum_{\emptyset \neq u \subseteq \{1,\ldots,s\}}\gamma_u U_{|u|,\alpha,t} \frac{(\log_p N)^{\alpha |u|}}{N^{\alpha}}, 
\end{equation*}
where $N=p^{m}+\cdots+p^{m+\alpha-1}$ and $U_{|u|,\alpha,t} >0$ for all $\emptyset \neq u \subseteq \{1,\ldots,s\}$.
\end{theorem}

In order to prove Theorem~\ref{thm:main}, we need some preparations. 
\begin{lemma}\label{lem:HO_net}
Let $P_{m, \alpha m}$ be an order $\alpha$ digital $(t,m,s)$-net over $\FF_p$ or be the first $p^m$ points of an order $\alpha$ digital $(t,s)$-sequence over $\FF_p$ such that $\alpha m>t$. For a non-empty subset $u\subseteq \{1,\ldots,s\}$, we write
\begin{equation*}
(P_{m, \alpha m})_u^{\perp} = \{\bsk_u\in \NN^{|u|} : (\bsk_u,\bszero) \in P_{m, \alpha m}^{\perp} \}. 
\end{equation*}
Then we have
\begin{equation*}
\sum_{\bsk_u\in (P_{m, \alpha m})_u^{\perp}}p^{-\mu_{\alpha}(\bsk_u)}\leq E_{|u|,\alpha}\frac{(\alpha m -t+2)^{\alpha |u|}}{p^{\alpha m-t}} ,
\end{equation*}
where
\begin{equation*}
E_{|u|,\alpha} = p^{\alpha |u|}\left( \frac{1}{p}+\left( \frac{p}{p-1}\right)^{\alpha|u|}\right).
\end{equation*}
\end{lemma}
\begin{proof}
See \citet[Lemma~5.2]{D08} and \citet[Lemma~15.20]{DPbook} for the proof.
\end{proof}

\begin{lemma}\label{lem:mu_alpha}
For $n\in \NN$ and $\bsk,\bsl\in \NN_0^s$ with $\bsk<p^n$, we have
\begin{equation*}
\mu_{\alpha}(\bsk+p^n \bsl) \geq \mu_{\alpha}(\bsk)+\mu_{\alpha}(\bsl).
\end{equation*}
\end{lemma}
\begin{proof}
Noting that 
\begin{equation*}
\mu_{\alpha}(\bsk+p^n \bsl) = \sum_{j=1}^s\mu_{\alpha}(k_j+p^nl_j)
\end{equation*}
and 
\begin{equation*}
\mu_{\alpha}(\bsk)+\mu_{\alpha}(\bsl) = \sum_{j=1}^s\left( \mu_{\alpha}(k_j)+\mu_{\alpha}(l_j) \right),
\end{equation*}
it suffices to prove the statement for the one-dimensional case:
\begin{equation*}
\mu_{\alpha}(k+p^n l) \geq \mu_{\alpha}(k)+\mu_{\alpha}(l), 
\end{equation*}
for any $k,l\in \NN_0$ with $k<p^n$. Since the result is trivial if $l=0$, we assume $l>0$. We denote the $p$-adic expansions of $k$ and $l$ by
\begin{align*}
k & = \kappa_1p^{c_1-1}+\cdots + \kappa_v p^{c_v-1}, \\
l & = \iota_1p^{d_1-1}+\cdots + \iota_w p^{d_w-1},
\end{align*}
with $\kappa_1,\ldots,\kappa_v,\iota_1,\ldots,\iota_w\in \{1,\ldots,p-1\}$, $c_1>\cdots>c_v>0$ and $d_1>\cdots>d_w>0$, respectively. Since $k<p^n$, we have $c_1\leq n$ and $v\leq n$. If $w< \alpha$, we have
\begin{align*}
 \mu_{\alpha}(k+p^n l) & = \mu_{\alpha}(\kappa_1p^{c_1-1}+\cdots + \kappa_v p^{c_v-1}+\iota_1p^{d_1+n-1}+\cdots + \iota_w p^{d_w+n-1}) \\
 & = \sum_{i=1}^{w}(d_i+n) + \sum_{i=1}^{\min(\alpha-w,v)}c_i \geq  \sum_{i=1}^{w}d_i + \sum_{i=1}^{\min(\alpha,v)}c_i  =  \mu_{\alpha}(l)+\mu_{\alpha}(k).
\end{align*}
On the other hand, if $w\geq \alpha$, we have
\begin{equation*}
 \mu_{\alpha}(k+p^n l) = \sum_{i=1}^{\alpha}(d_i+n) \geq \sum_{i=1}^{\alpha}d_i+\sum_{i=1}^{\min(\alpha,v)}c_i  =  \mu_{\alpha}(l)+\mu_{\alpha}(k).
\end{equation*}
Thus we complete the proof.
\end{proof}

Now we are ready to prove Theorem~\ref{thm:main}.

\begin{proof}[Proof of Theorem~\ref{thm:main}]
Let $f\in W_{s,\alpha,q,r}$. For each $0\leq i<\alpha$, Corollary~\ref{cor:euler-mac} gives
\begin{equation*}
I_{m+i}^{(1)}(f) = I(f)+\sum_{\bsl\in \NN_0^s}\sum_{\substack{\bsk\in (P^{[(m+i)\times (m+i)]})^{\perp}\\ \bszero\neq \bsk<p^{m+i}}}\hat{f}(\bsk+p^{m+i}\bsl)+\sum_{\tau=1}^{\alpha-1}\frac{c_\tau(f)}{p^{\tau (m+i)}}+R_{s,\alpha,m+i} .
\end{equation*}
Using the result shown in \citet[Lemma~2.9 \& Corollary~2.11]{DGYxx}, we have
\begin{equation*}
I_{m}^{(\alpha)}(f) = I(f) + \sum_{i=0}^{\alpha-1}w_i\left( \sum_{\bsl\in \NN_0^s}\sum_{\substack{\bsk\in (P^{[(m+i)\times (m+i)]})^{\perp}\\ \bszero\neq \bsk<p^{m+i}}}\hat{f}(\bsk+p^{m+i}\bsl)+R_{s,\alpha,m+i}\right), 
\end{equation*}
with
\begin{equation*}
w_i = \prod_{j=1}^{\alpha-i-1}\left( \frac{-1}{p^j-1}\right)\prod_{j=1}^{i}\left( \frac{p^j}{p^j-1}\right) \quad \text{for $0\leq i\leq \alpha-1$},
\end{equation*}
where the empty product is set to 1. Here we note that $\sum_{i=0}^{\alpha-1}w_i=1$, see \citet[Lemma~2.10]{DGYxx}. It follows from the triangle inequality and the decomposition
\begin{equation*}
(P^{[(m+i)\times (m+i)]})^{\perp}\setminus \{\bszero\} = \bigcup_{\emptyset \neq u\subseteq \{1,\ldots,s\}}(P^{[(m+i)\times (m+i)]})^{\perp}_u, 
\end{equation*}
that 
\begin{align*}
& \left| I_{m}^{(\alpha)}(f) -I(f)\right| \\
& \leq \sum_{i=0}^{\alpha-1}|w_i| \left( \sum_{\bsl\in \NN_0^s}\sum_{\substack{\bsk\in (P^{[(m+i)\times (m+i)]})^{\perp}\\ \bszero\neq \bsk<p^{m+i}}}| \hat{f}(\bsk+p^{m+i}\bsl)|+|R_{s,\alpha,m+i}|\right) \\
 & = \sum_{i=0}^{\alpha-1}|w_i| \left( \sum_{\emptyset \neq u\subseteq \{1,\ldots,s\}}\sum_{v\subseteq u}\sum_{\bsl_v\in \NN^{|v|}}\sum_{\substack{(\bsk_v,\bsk_{u\setminus v})\in (P^{[(m+i)\times (m+i)]})^{\perp}_u\\ \bszero\neq \bsk_v<p^{m+i}\\ 0< \bsk_{u\setminus v}<p^{m+i}}}| \hat{f}(\bsk_v+p^{m+i}\bsl_v, \bsk_{u\setminus v},\bszero)|+|R_{s,\alpha,m+i}|\right) \\
 & \leq \sum_{i=0}^{\alpha-1}|w_i| \left( \sum_{\emptyset \neq u\subseteq \{1,\ldots,s\}}\sum_{\bsl_u\in \NN_0^{|u|}}\sum_{\substack{\bsk_u\in (P^{[(m+i)\times (m+i)]})^{\perp}_u\\ \bszero \neq \bsk_u<p^{m+i}}}| \hat{f}(\bsk_u+p^{m+i}\bsl_u,\bszero)|+|R_{s,\alpha,m+i}|\right) .
\end{align*}

Using Lemmas~\ref{lem:walsh_decay} and \ref{lem:mu_alpha}, the inner double sum above for a given $\emptyset \neq u\subseteq \{1,\ldots,s\}$ is bounded by
\begin{align*}
\sum_{\bsl_u\in \NN_0^{|u|}}\sum_{\substack{\bsk_u\in (P^{[(m+i)\times (m+i)]})^{\perp}_u\\ \bszero \neq \bsk_u<p^{m+i}}}| \hat{f}(\bsk_u+p^{m+i}\bsl_u,\bszero)| & \leq \gamma_u \| f_u\|_{s,\alpha,q,r} C_{\alpha}^{|u|}\sum_{\bsl_u\in \NN_0^{|u|}}\sum_{\substack{\bsk_u\in (P^{[(m+i)\times (m+i)]})^{\perp}_u\\ \bsk_u<p^{m+i}}}p^{-\mu_{\alpha}(\bsk_u+p^{m+i}\bsl_u)} \\
& \leq \gamma_u \| f\|_{s,\alpha,q,r} C_{\alpha}^{|u|}\sum_{\bsl_u\in \NN_0^{|u|}}p^{-\mu_{\alpha}(\bsl_u)} \sum_{\substack{\bsk_u\in (P^{[(m+i)\times (m+i)]})^{\perp}_u\\ \bsk_u<p^{m+i}}}p^{-\mu_{\alpha}(\bsk_u)} .
\end{align*}
Applying \eqref{eq:trunc_dual} and Lemma~\ref{lem:HO_net}, the inner sum over $\bsk_u$ is bounded by
\begin{align*}
\sum_{\substack{\bsk_u\in (P^{[(m+i)\times (m+i)]})^{\perp}_u\\ \bsk_u<p^{m+i}}}p^{-\mu_{\alpha}(\bsk_u)} & = \sum_{\substack{\bsk_u\in (P^{[(m+i)\times (m+i)]})^{\perp}_u\\ \bsk_u<p^{m+i}}}p^{-\mu_{\alpha}(\bsk_u)} \leq \sum_{\bsk_u\in (P^{[\NN\times (m+i)]})^{\perp}_u}p^{-\mu_{\alpha}(\bsk_u)} \\
 & \leq E_{|u|,\alpha}\frac{(\alpha (m+i) -t+2)^{\alpha |u|}}{p^{\alpha(m+i)-t}} .
\end{align*}
Regarding the sum over $\bsl_u\in \NN_0^{|u|}$, \citet[Lemma~7]{G16} gives
\begin{equation*}
\sum_{\bsl_u\in \NN_0^{|u|}}p^{-\mu_{\alpha}(\bsl_u)} = \left( \sum_{l\in \NN_0}p^{-\mu_{\alpha}(l)}\right)^{|u|} = A_{\alpha}^{|u|}, 
\end{equation*}
where
\begin{equation*}
A_{\alpha} = 1+\sum_{w=1}^{\alpha-1}\prod_{i=1}^{w}\left( \frac{p-1}{p^i-1}\right)+ \left(\frac{p^{\alpha}-1}{p^{\alpha}-p}\right)\prod_{i=1}^{\alpha}\left( \frac{p-1}{p^i-1}\right). 
\end{equation*}
All together, the inner double sum for $\emptyset \neq u\subseteq \{1,\ldots,s\}$ is bounded by
\begin{equation*}
\sum_{\bsl_u\in \NN_0^{|u|}}\sum_{\substack{\bsk_u\in (P^{[(m+i)\times (m+i)]})^{\perp}_u\\ 0<\bsk_u<p^{m+i}}}| \hat{f}(\bsk_u+p^{m+i}\bsl_u,\bszero)| \leq \gamma_u \| f\|_{s,\alpha,q,r} A_{\alpha}^{|u|}C_{\alpha}^{|u|}E_{|u|,\alpha}\frac{(\alpha (m+i) -t+2)^{\alpha |u|}}{p^{\alpha(m+i)-t}} .
\end{equation*}

Recall that the total number of function evaluations is $N=p^m+\cdots+p^{m+\alpha-1}$. Using the above result and \eqref{eq:remainder_bound}, the integration error for $f\in W_{s,\alpha,q,r}$ is bounded by
\begin{align}
\left| I_{m}^{(\alpha)}(f) -I(f)\right|  & \leq  \sum_{i=0}^{\alpha-1}|w_i| \sum_{\emptyset \neq u\subseteq \{1,\ldots,s\}}\gamma_u \| f\|_{s,\alpha,q,r} \frac{p^tA_{\alpha}^{|u|}C_{\alpha}^{|u|}E_{|u|,\alpha}(\alpha (m+i) -t+2)^{\alpha |u|}+(\alpha+1)^{|u|}D_{\alpha}^{|u|}}{p^{\alpha(m+i)}} \nonumber \\
 & \leq \| f\|_{s,\alpha,q,r}\sum_{\emptyset \neq u\subseteq \{1,\ldots,s\}}\gamma_u \frac{\left( \log_p N\right)^{\alpha|u|}}{N^\alpha} \nonumber \\
 & \quad \times \left( p^t A_{\alpha}^{|u|}C_{\alpha}^{|u|}E_{|u|,\alpha}+(\alpha+1)^{|u|}D_{\alpha}^{|u|}\right)\sum_{i=0}^{\alpha-1}|w_i|\frac{N^\alpha}{p^{\alpha(m+i)}}\frac{(\alpha (m+i) -t+2)^{\alpha |u|}}{\left( \log_p N\right)^{\alpha|u|}}.\label{eq:error_bound_proof_1}
\end{align}
For any $0\leq i<\alpha$ we have
\begin{equation*}
\frac{N^\alpha}{p^{\alpha(m+i)}}\leq \frac{(\alpha p^{m+\alpha-1})^\alpha}{p^{\alpha(m+i)}}=\left( \alpha p^{\alpha+1-i}\right)^{\alpha},
\end{equation*}
and 
\begin{align*}
 (\alpha (m+i) -t+2)^{\alpha |u|} & \leq \left(\alpha (m+i+2/\alpha)\right)^{\alpha|u|} \leq \left(\alpha (m+i+1)\right)^{\alpha|u|} \\
 & \leq \left(2\alpha (m+i)\right)^{\alpha|u|} \leq \left(2\alpha \log_p N\right)^{\alpha|u|}.
\end{align*}
Thus the inner sum of \eqref{eq:error_bound_proof_1} is bounded independently of $m$ as
\begin{equation*}
\sum_{i=0}^{\alpha-1}|w_i|\frac{N^\alpha}{\left( \log_p N\right)^{\alpha|u|}}\frac{(\alpha (m+i) -t+2)^{\alpha |u|}}{p^{\alpha(m+i)}}\leq \left(2\alpha \right)^{\alpha|u|}\sum_{i=0}^{\alpha-1}|w_i|\left( \alpha p^{\alpha+1-i}\right)^{\alpha}.
\end{equation*}
This leads to a worst-case error bound:
\begin{equation*}
\sup_{\substack{f\in W_{s,\alpha,q,r}\\ \|f\|_{s,\alpha,q,r}\leq 1}}\left| I_{m}^{(\alpha)}(f) -I(f)\right| \leq \sum_{\emptyset \neq u\subseteq \{1,\ldots,s\}}\gamma_u U_{|u|,\alpha,t}\frac{\left( \log_p N\right)^{\alpha|u|}}{N^\alpha},
\end{equation*}
where
\begin{equation*}
U_{|u|,\alpha,t} = \left(2\alpha \right)^{\alpha|u|}\left( p^t A_{\alpha}^{|u|}C_{\alpha}^{|u|}E_{|u|,\alpha}+(\alpha+1)^{|u|}D_{\alpha}^{|u|}\right) \sum_{i=0}^{\alpha-1}|w_i|\left( \alpha p^{\alpha+1-i}\right)^{\alpha}.
\end{equation*}
Hence we complete the proof.
\end{proof}

\begin{remark}\label{rem:extensible}
Algorithm~\ref{alg:first} is naturally extensible with respect to $m$ in the following way: For $m_{\min}, m_{\max}\in \NN$, $m_{\max}-m_{\min}\geq \alpha$ and $f\colon [0,1)^s\to \RR$, do the following:
\begin{enumerate}
\item For $m_{\min}\leq  i\leq m_{\max}$, compute
\begin{equation*}
I_{i}^{(1)}(f) := I\left(f;P^{[i\times i]}\right).
\end{equation*}
\item For $1\leq \tau<\alpha$, let
\begin{equation*}
I_{i}^{(\tau+1)}(f) := \frac{p^\tau I_{i+1}^{(\tau)}(f)-I_{i}^{(\tau)}(f)}{p^\tau-1}\quad \text{for $m_{\min}\leq i\leq m_{\max}-\tau$}.
\end{equation*}
\end{enumerate}
Then we obtain a sequence of the approximate values $I_{m_{\min}}^{(\alpha)}, I_{m_{\min}+1}^{(\alpha)},\ldots,I_{m_{\max}-\alpha+1}^{(\alpha)}$. If one wants to increase $m_{\max}$ by 1, it suffices to compute $I_{m_{\max}+1}^{(1)}(f)$ instead of whole $I_{m_{\max}-\alpha+1}^{(1)}(f),\ldots,I_{m_{\max}+1}^{(1)}(f)$. This is a key advantage as compared to another possible algorithm which we introduce below.
\end{remark}

\subsection{Another possible algorithm}
It is clear from the proof of Theorem~\ref{thm:main} that, in order to vanish the main terms of \eqref{eq:euler-mac}, i.e.,
\begin{equation*}
\sum_{\tau=1}^{\alpha-1}\frac{c_\tau(f)}{p^{\tau n}} 
\end{equation*}
by applying Richardson extrapolation recursively, there is no need to set $n=m$ and to change them at the same time. In fact we can fix $m$ and change $n$ only, although the resulting algorithm is no longer extensible in $m$.
\begin{algorithm}\label{alg:second}
Let $P_{m,\alpha m}$ be an order $\alpha$ digital $(t,m,s)$-net over $\FF_p$ with generating matrices $C_1,\ldots,C_s\in \FF_p^{\alpha m\times m}$. For $f\colon [0,1)^s\to \RR$, do the following:
\begin{enumerate}
\item For $0\leq i<\alpha$, compute
\begin{equation*}
J_{m+i}^{(1)} := I\left(f;P_{m,\alpha m}^{[(m+i)\times m]}\right),
\end{equation*}
where $P_{m,\alpha m}^{[(m+i)\times m]}$ denotes a digital net with generating matrices $C_1^{[(m+i)\times m]}, \ldots, C_s^{[(m+i)\times m]}$.
\item For $1\leq \tau<\alpha$, let
\begin{equation*}
J_{m+i}^{(\tau+1)} := \frac{p^\tau J_{m+i+1}^{(\tau)}-J_{m+i}^{(\tau)}}{p^\tau-1}\quad \text{for $0\leq i<\alpha-\tau$}.
\end{equation*}
\item Return $J_{m}^{(\alpha)}$ as an approximation of $I(f)$.
\end{enumerate}
\end{algorithm}
We see that the total number of function evaluations used in $J_{m}^{(\alpha)}$ is $N=\alpha p^m$. Similarly to Algorithm~\ref{alg:first}, this alternative algorithm achieves the almost optimal rate of convergence as shown below. Since we can prove the result exactly in the same way as Theorem~\ref{thm:main}, we omit the proof.
\begin{theorem}\label{thm:main2}
Let $\alpha\in \NN$, $\alpha\geq 2$, and $1\leq q,r \leq \infty$. The worst-case error of the algorithm $J_m^{(\alpha)}(f)$ in $W_{s,\alpha,q,r}$ is bounded above by
\begin{equation*}
\sup_{\substack{f\in W_{s,\alpha,q,r}\\ \|f\|_{s,\alpha,q,r}\leq 1}}\left| I(f)-J_m^{(\alpha)}(f)\right| \leq \sum_{\emptyset \neq u \subseteq \{1,\ldots,s\}}\gamma_u V_{|u|,\alpha,t} \frac{(\log_p N)^{\alpha |u|}}{N^{\alpha}},
\end{equation*}
where $N=\alpha p^m$ and 
\begin{equation*}
V_{|u|,\alpha,t} = \left(\frac{\alpha}{\log_p 2} \right)^{\alpha|u|}\left( p^t A_{\alpha}^{|u|}C_{\alpha}^{|u|}E_{|u|,\alpha}+(\alpha+1)^{|u|}D_{\alpha}^{|u|}\right) \sum_{i=0}^{\alpha-1}|w_i|.
\end{equation*}
for all $\emptyset \neq u \subseteq \{1,\ldots,s\}$.
\end{theorem}

\section{Numerical experiments}\label{sec:exp}
Finally we conduct some numerical experiments to confirm the effectiveness of our extrapolation-based quadrature rules. For all the experiments, we set the base $p=2$ and use a MATLAB implementation of higher order Sobol' nets and sequences from \citet{D07,D08}.

\subsection{Low-dimensional cases}
First let us consider the simplest case $s=1$. The test function we use is
\begin{equation*}
f_1(x)=x^3\left( \log x+\frac{1}{4}\right).
\end{equation*}
While the third derivative of $f_1$ is in $L_q([0,1))$ for any finite $q\geq 1$, the fourth derivative is not in $L_1([0,1))$, implying that $f_1\not\in W_{1,4,q,r}$ but $f_1\in W_{1,3,q,r}$.  Note that $I(f_1)=0$. Figure~\ref{fig:test1} shows the absolute integration error obtained by using Algorithm~\ref{alg:first} (and Remark~\ref{rem:extensible}) with $\alpha=2$ and $\alpha=3$. In both cases, the integration error of $I_m^{(1)}$ achieves the convergence of nearly order $N^{-1}$. We can see that the order of convergence of the integration error is improved from $N^{-1}$ to $N^{-2}$ by applying Richardson extrapolation. In case of $\alpha=3$, the recursive application of Richardson extrapolation further improves the order of convergence to approximately $N^{-3}$. This convergence behavior is in good agreement with our theoretical result.
\begin{figure}[t!]
\centering
\includegraphics[width=6cm]{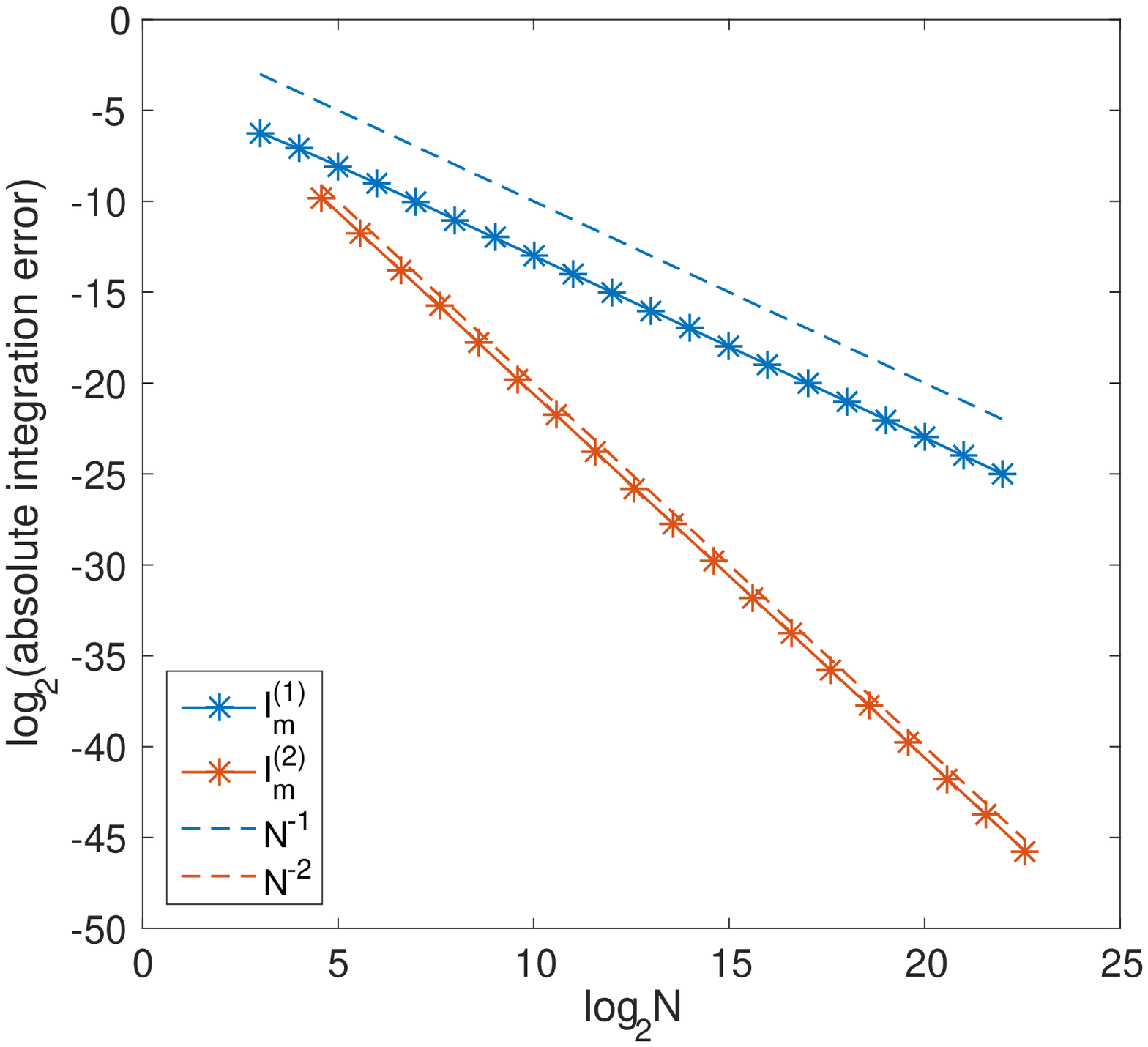}
\includegraphics[width=6cm]{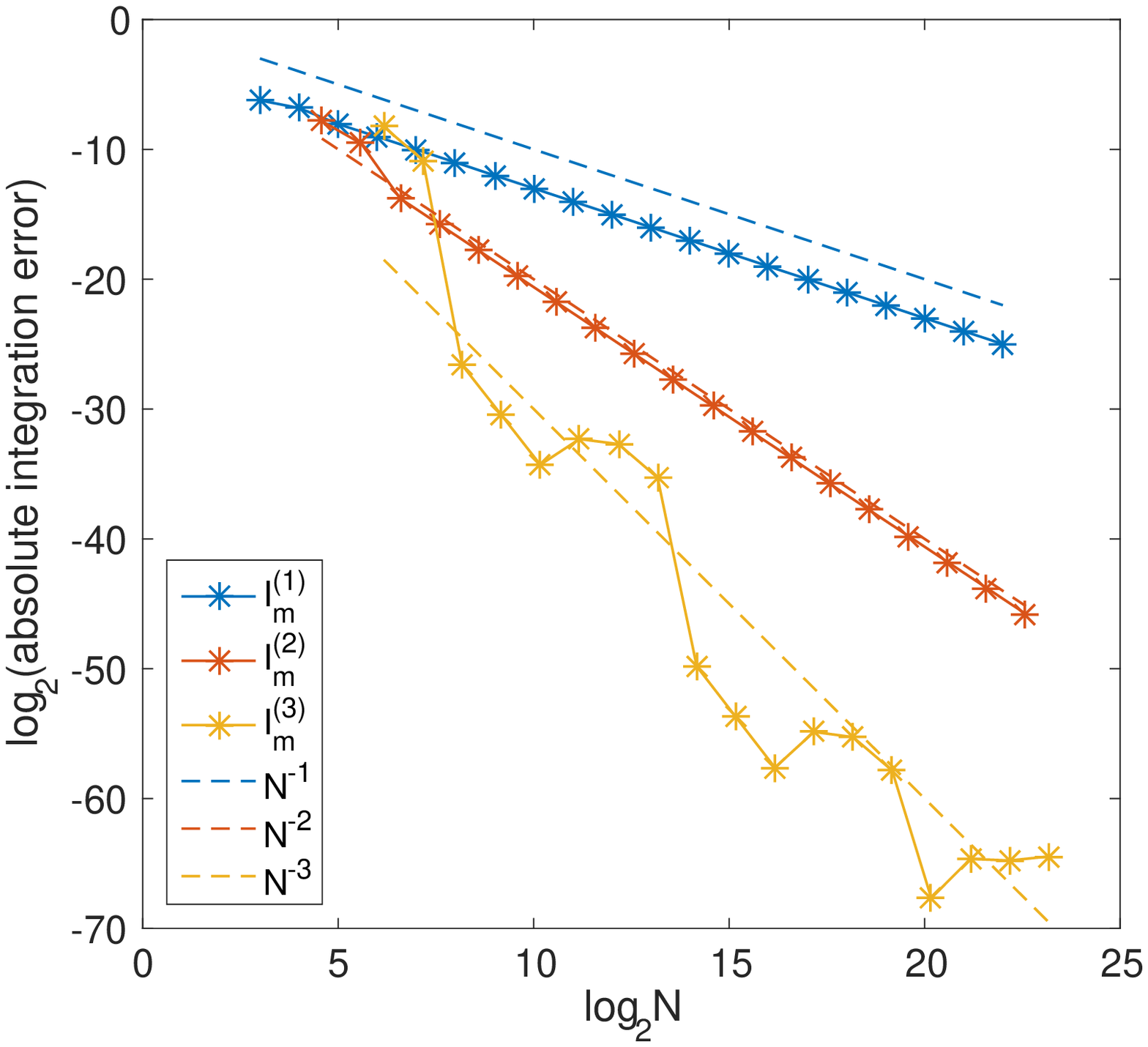}
\caption{Integration error for $f_1$: $\alpha=2$ (left) and $\alpha=3$ (right).}
\label{fig:test1}
\end{figure}

Next let us consider a bi-variate test function
\begin{equation*}
f_2(x,y) = \left( \frac{1}{2}-xy\right)^6 \bsone_{xy\leq 1/2}, 
\end{equation*}
where $\bsone_A$ denotes the indicator function of an event $A$. The derivative $f_2^{(3,3)}$ has a discontinuity along the curve $xy=1/2$ but is in $L_q([0,1)^2)$ for any $q$, which ensures $f_2\in W_{2,3,q,r}$. Note that we have 
\begin{equation*}
I(f_2) = \frac{1}{896}\left( \frac{363}{140}+\log 2\right).
\end{equation*}
Figure~\ref{fig:test2} shows the absolute integration error by Algorithm~\ref{alg:first} (and Remark~\ref{rem:extensible}) with $\alpha=2$ and $\alpha=3$. Similarly to the result for $f_1$, the integration errors of $I_m^{(1)}$ and $I_m^{(2)}$ achieve the convergence of nearly order $N^{-1}$ and $N^{-2}$, respectively. For the case $\alpha=3$, after the recursive application of Richardson extrapolation, the error decays asymptotically with the order $N^{-3}$. However, the magnitude of the error itself is almost comparable to that for $I_m^{(2)}$ in this range of $N$. In fact, as can be seen from the right plot of Figure~\ref{fig:test1-2_interlaced}, QMC rules using order 3 Sobol' sequences achieve the convergence of order $N^{-3}$ only asymptotically, and the performances of order 2 and 3 Sobol' sequences are comparable. This implies that, on the right-hand side of \eqref{eq:euler-mac}, $c_1(f)/p^n$ is the most dominant term, but $c_2(f)/p^{2n}$ is not the only secondary dominant term and is comparable to
\[ \sum_{\bsl\in \NN_0^s}\sum_{\substack{\bsk\in P_{m,n}^{\perp}\\ \bszero\neq \bsk<p^n}}\hat{f}(\bsk+p^n\bsl) .\]
This is why our Algorithm~\ref{alg:first} cannot achieve the desired rate of convergence for $\alpha=3$ when $m$ is not large enough. Thus, improving the performance of original higher order digital nets and sequences is important for our extrapolation-based rules to work well.
\begin{figure}[t!]
\centering
\includegraphics[width=6cm]{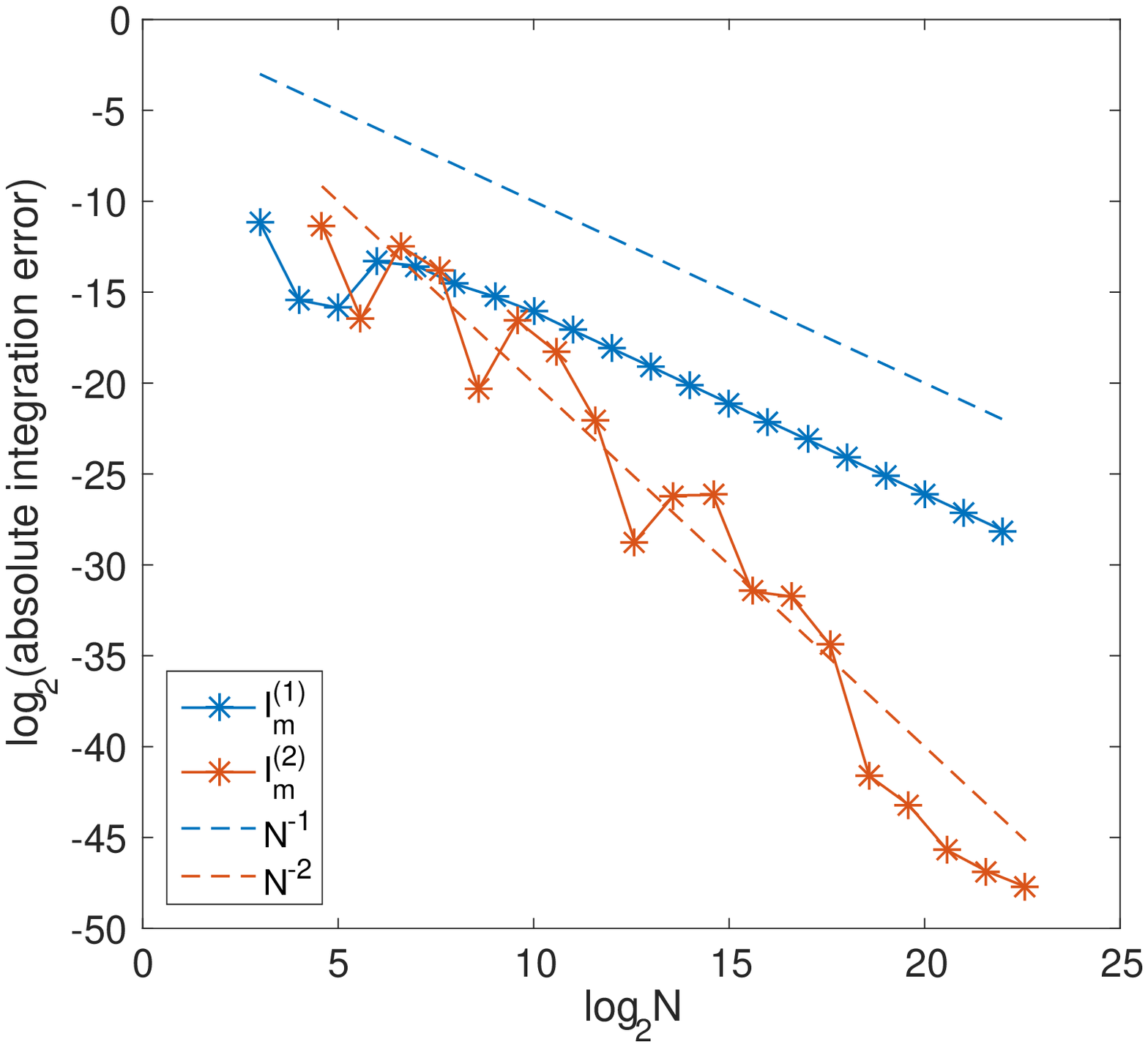}
\includegraphics[width=6cm]{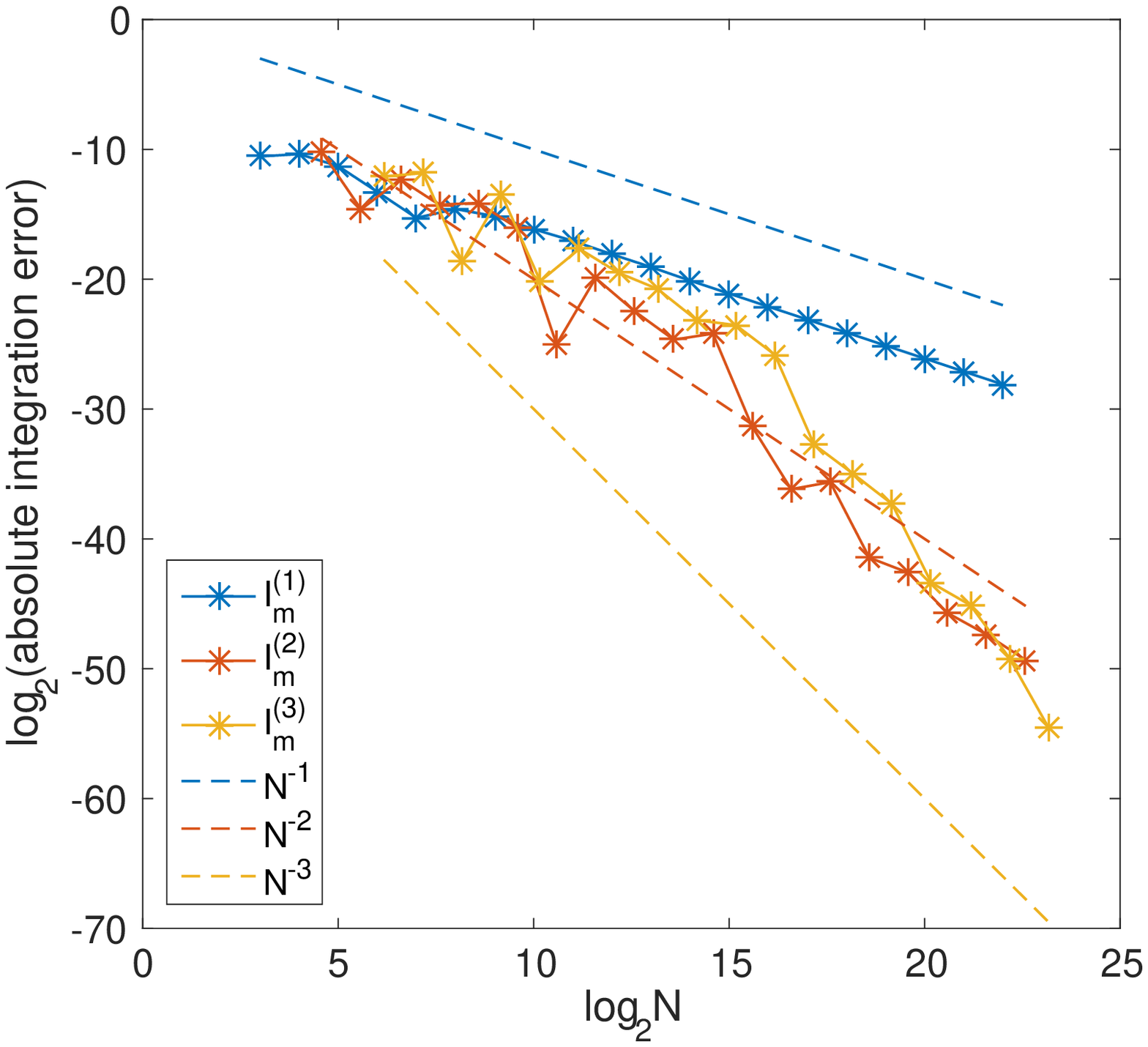}
\caption{Integration error for $f_2$: $\alpha=2$ (left) and $\alpha=3$ (right).}
\label{fig:test2}
\end{figure}
\begin{figure}[t!]
\centering
\includegraphics[width=6cm]{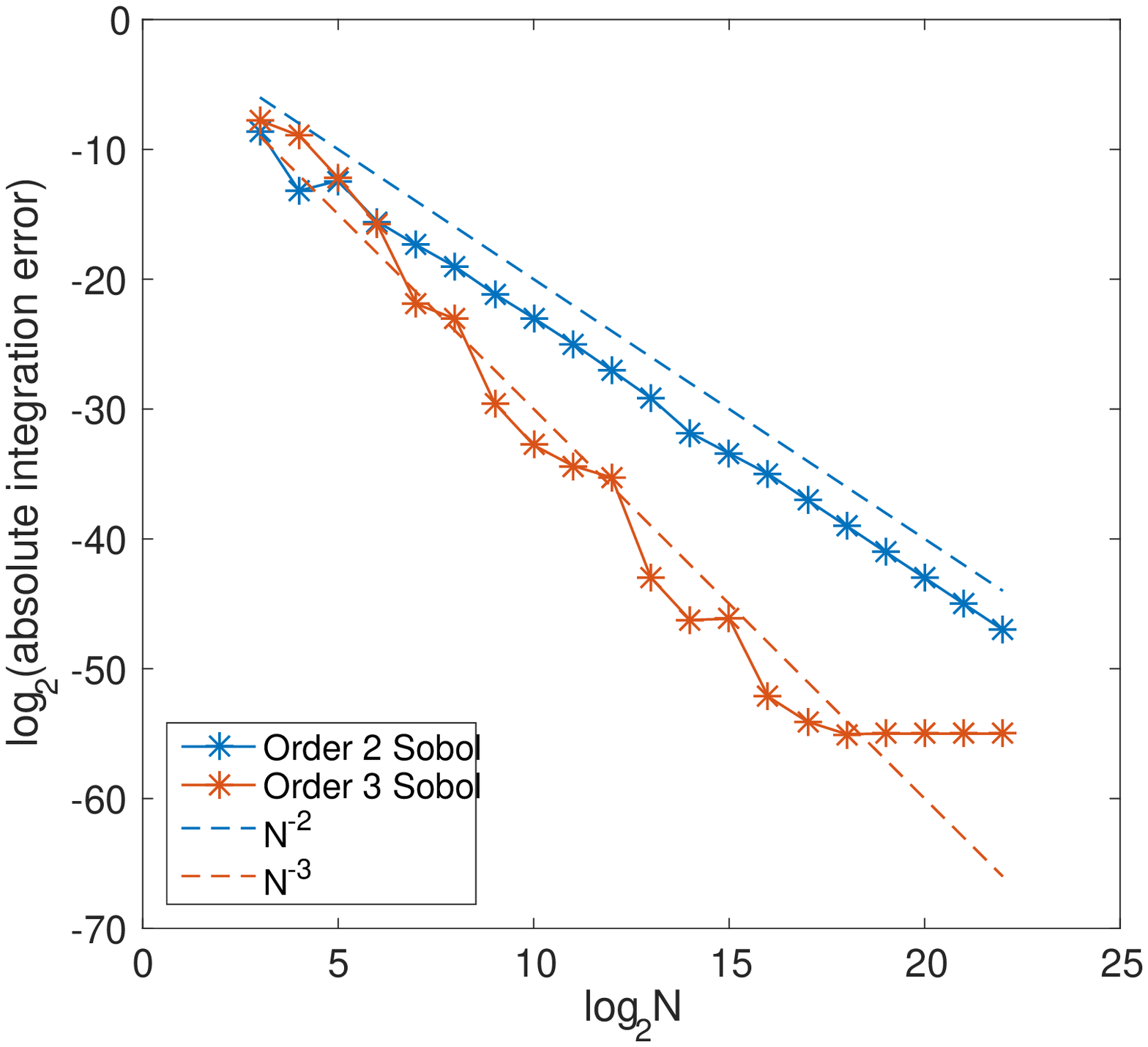}
\includegraphics[width=6cm]{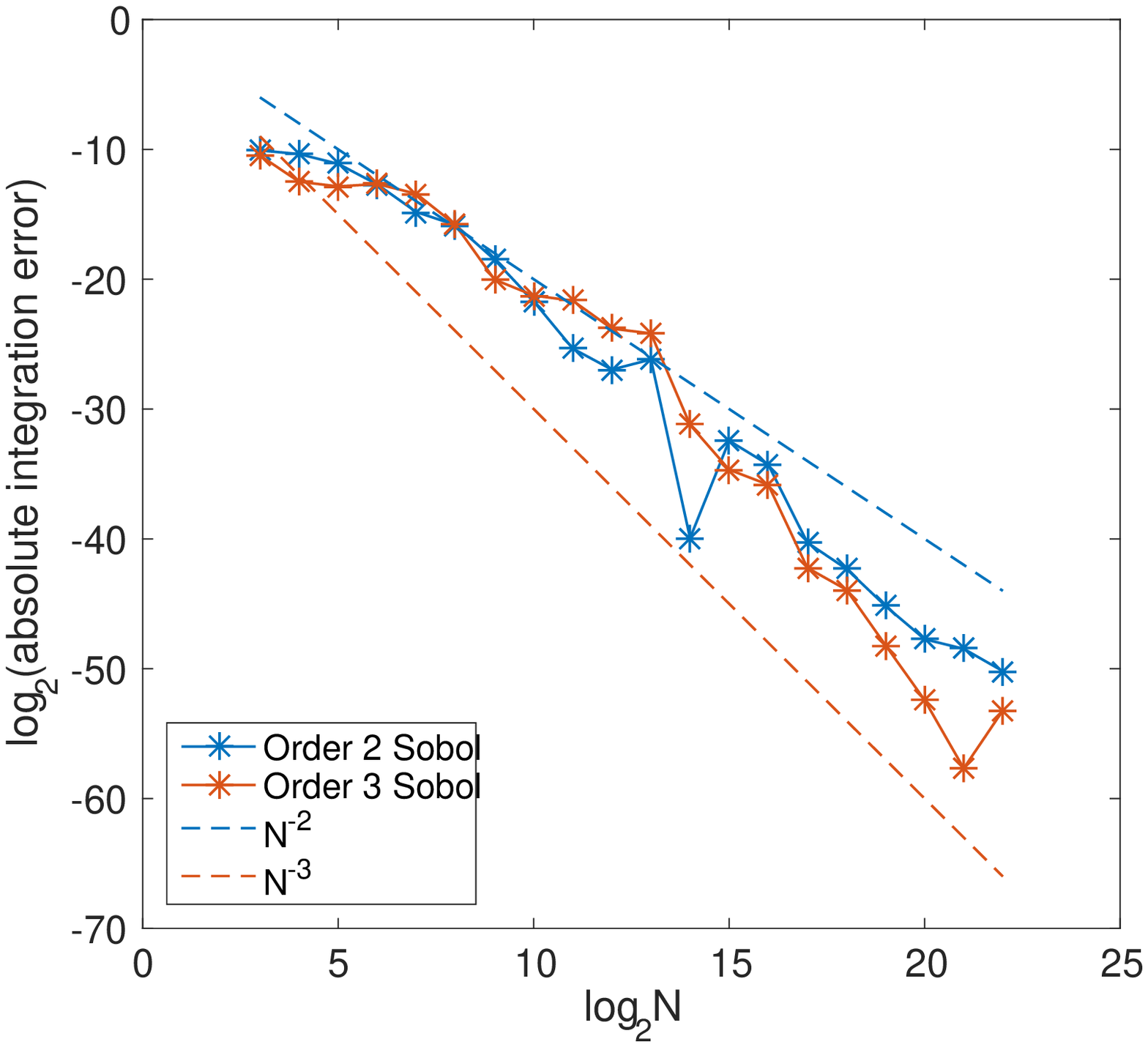}
\caption{Integration error by QMC rules using order 2 and 3 Sobol' sequences for $f_1$ (left) and $f_2$ (right). If $\alpha m > 52$, the truncation map $\tr_{52}$ is applied to all of the quadrature nodes.}
\label{fig:test1-2_interlaced}
\end{figure}
\begin{figure}[t!]
\centering
\includegraphics[width=6cm]{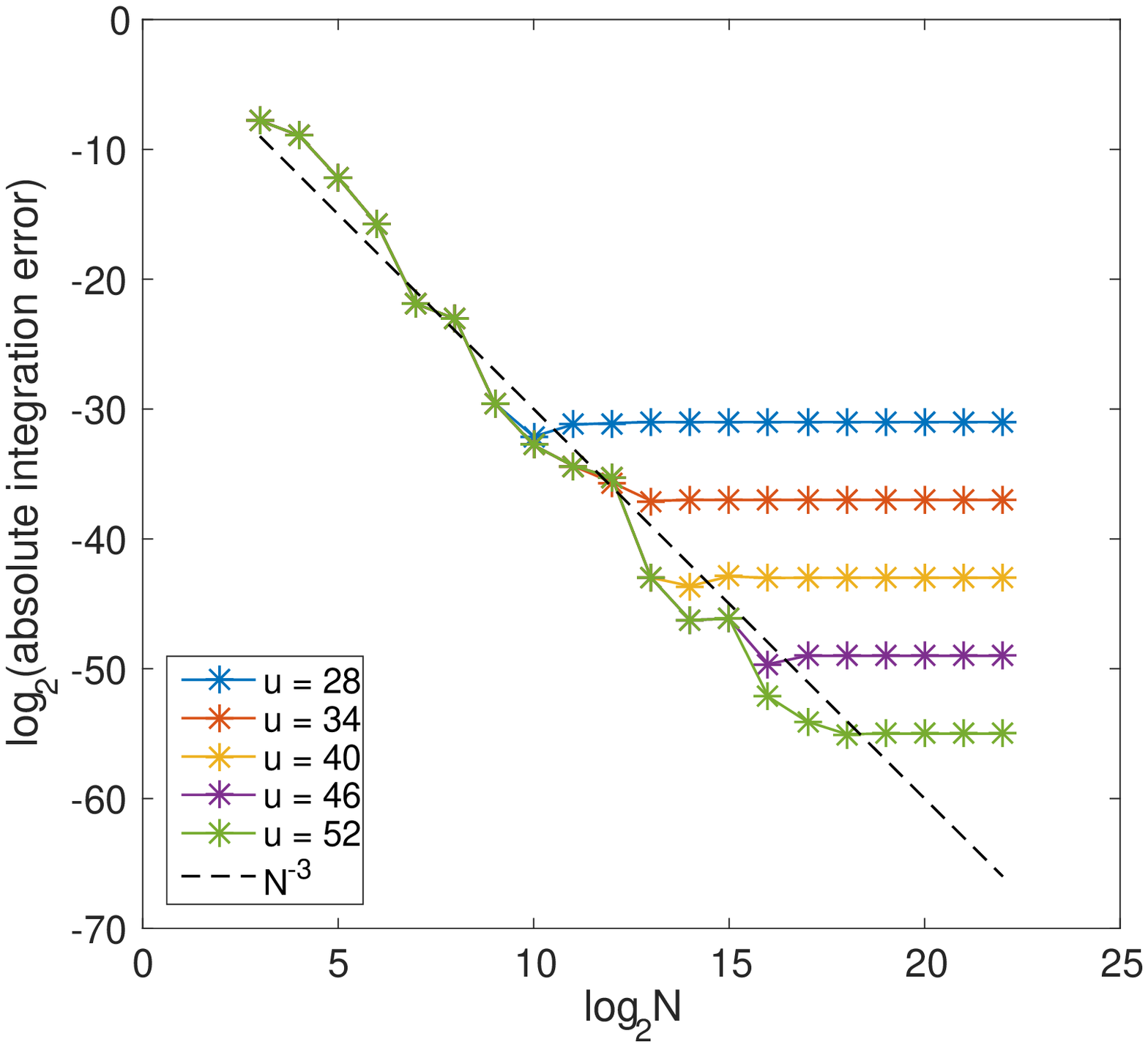}
\includegraphics[width=6cm]{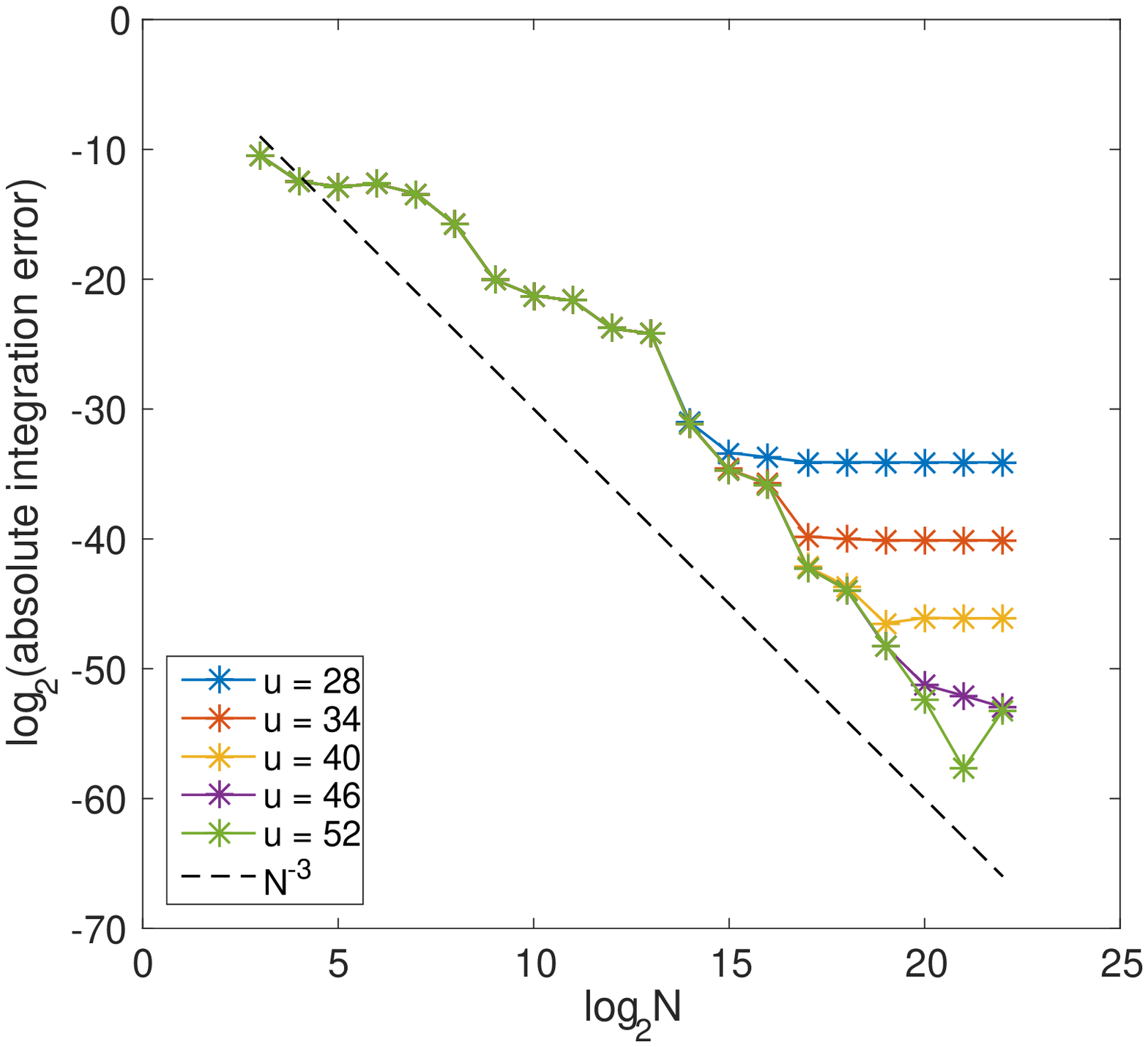}
\caption{Effect of precision $u$ for QMC rules using order 3 Sobol' sequences for $f_1$ (left) and $f_2$ (right). If $\alpha m > u$, the truncation map $\tr_{u}$ is applied to all of the quadrature nodes.}
\label{fig:test1-2_precision}
\end{figure}

When IEEE 754 double-precision floating-point format is employed, the first $2^m$ points of order 2 Sobol' sequences can be represented with full precision in this range of $m\in [3,22]$, whereas those of order 3 Sobol' sequences cannot for $m\geq 18$. As we see from Figure~\ref{fig:test1-2_interlaced}, the integration error for $f_1$ using order 3 Sobol' sequences remains almost the same for $m\geq 18$, which is considered to be the consequence of rounding the quadrature nodes. In fact, by changing the maximum precision from $52$ to lower values $u$, the switch of the convergence behavior from the $O(N^{-3})$ decay to the plateau happens for smaller $m$ as shown in the left plot of Figure~\ref{fig:test1-2_precision}. Since our extrapolation-based quadrature rules do not suffer from the rounding problem in this range of $m$, the error continues to decay even for $m\geq 18$ as is clear from the right plot of Figure~\ref{fig:test1}. However, such a switch of the convergence behavior for order 3 Sobol' sequences cannot be clearly observed for $f_2$, see the right plot of Figure~\ref{fig:test1-2_interlaced}. Changing the maximum precision from $52$ to lower values $u$ does yield such a switch but for larger $m$ as compared to the case for $f_1$. Hence whether or not the round-off error is comparable to the integration error when $\alpha m$ goes beyond an available precision depends on an integrand, and in general, it seems quite difficult to distinguish the round-off error from the integration error. Again we would like to emphasize that our extrapolation-based quadrature rules are free from such difficulty unless $m$ is large enough, say $m=52$, which is an important advantage compared to the original higher order digital nets and sequences.

\subsection{High-dimensional cases}
Let us move on to the high-dimensional setting. Following \citet{DGYxx} and \citet{GS16}, we consider the following two test functions:
\begin{align*}
f_3(\bsx) & = \prod_{j=1}^{s}\left[ 1+\gamma_j \left( x_j^{c_1}-\frac{1}{1+c_1}\right)\right], \\
f_4(\bsx) & = \exp\left( c_2\sum_{j=1}^{s}\gamma_j x_j\right), 
\end{align*}
with parameters $c_1>0,c_2\neq 0$, for which we have $I(f_3)=1$ and
\begin{equation*}
I(f_4)  = \prod_{j=1}^{s}\left[ \frac{1}{c_2\gamma_j}\left(\exp( c_2\gamma_j)-1\right)\right].
\end{equation*}
When $1<c_1<2$, the second derivative of the function $x\mapsto x^{c_1}$ is not absolutely continuous but is in $L_q([0,1))$ for any $q<1/(2-c_1)$, which means that $f_3\in W_{s,2,q,r}$ for $q<1/(2-c_1)$. On the other hand, $f_4$ is analytic and belongs to $W_{s,\alpha,q,r}$ for any $\alpha \geq 2$ and $q\geq 1$. As stated by \citet{GS16}, $f_4$ is designed to mimic the behavior of parametric solution families of partial differential equations. We employ this test function to see potential applicability to such problems. 

We put $s=100$ and $\gamma_j=j^{-2}$. We consider three quadrature rules: Algorithm~\ref{alg:first} with $\alpha=2$, denoted by $I_m^{(2)}$, Algorithm~\ref{alg:second} with $\alpha=2$, denoted by $J_m^{(2)}$, and QMC rules using order 2 Sobol' sequences. Figure~\ref{fig:test3-5} shows the comparison of the absolute integration errors obtained by these three algorithms. In fact, there is no decisive difference in performance between these algorithms, and all of them achieve the nearly desired rate of convergence, which is $O(N^{-2+\varepsilon})$ for arbitrarily small $\varepsilon>0$. This result not only supports our theoretical result, but also indicates that Richardson extrapolation allows truncation of higher order digital nets and sequences without sacrificing the practical performance of them even for high-dimensional cases.
\begin{figure}[t!]
\centering
\includegraphics[width=6cm]{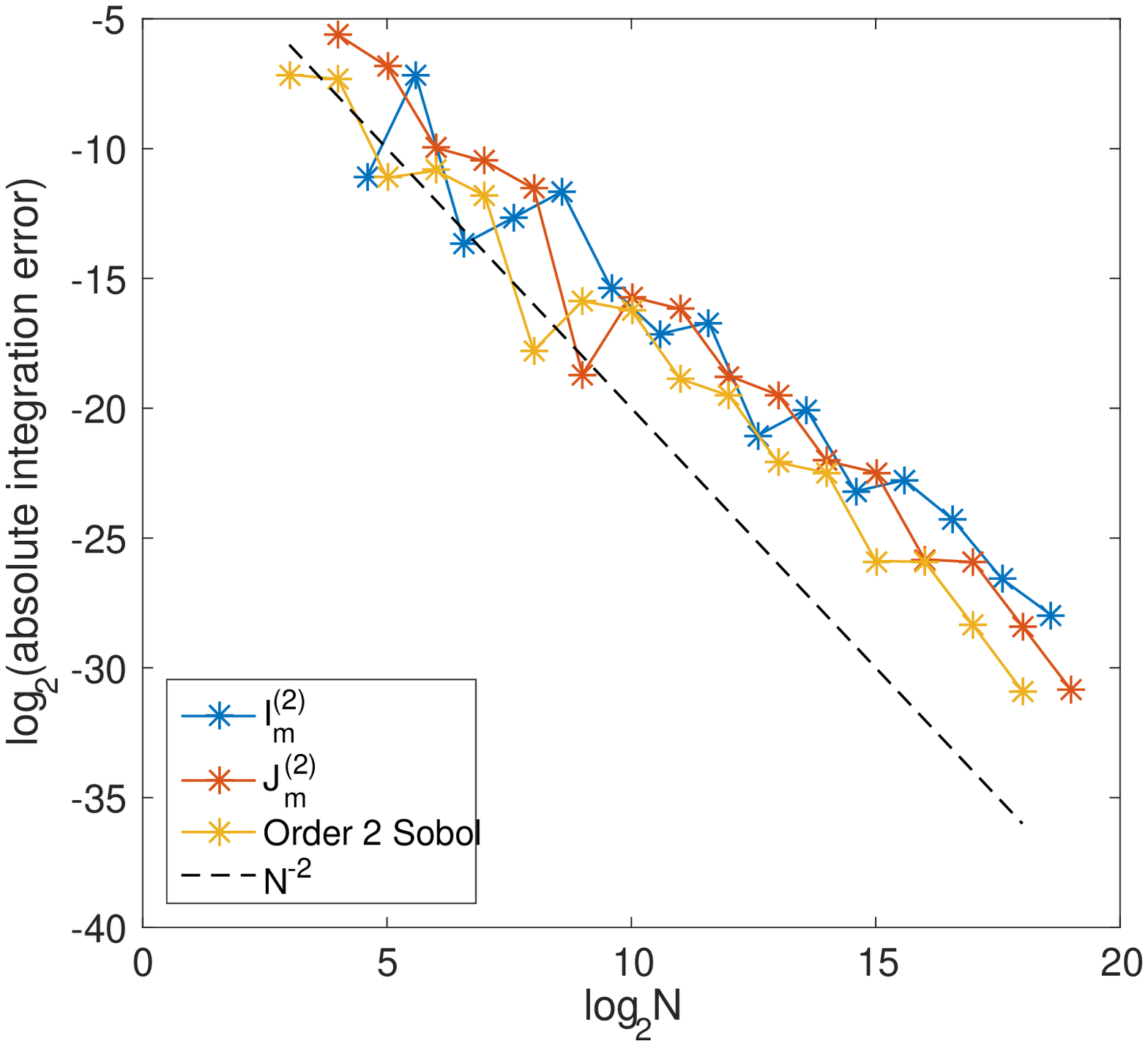}
\includegraphics[width=6cm]{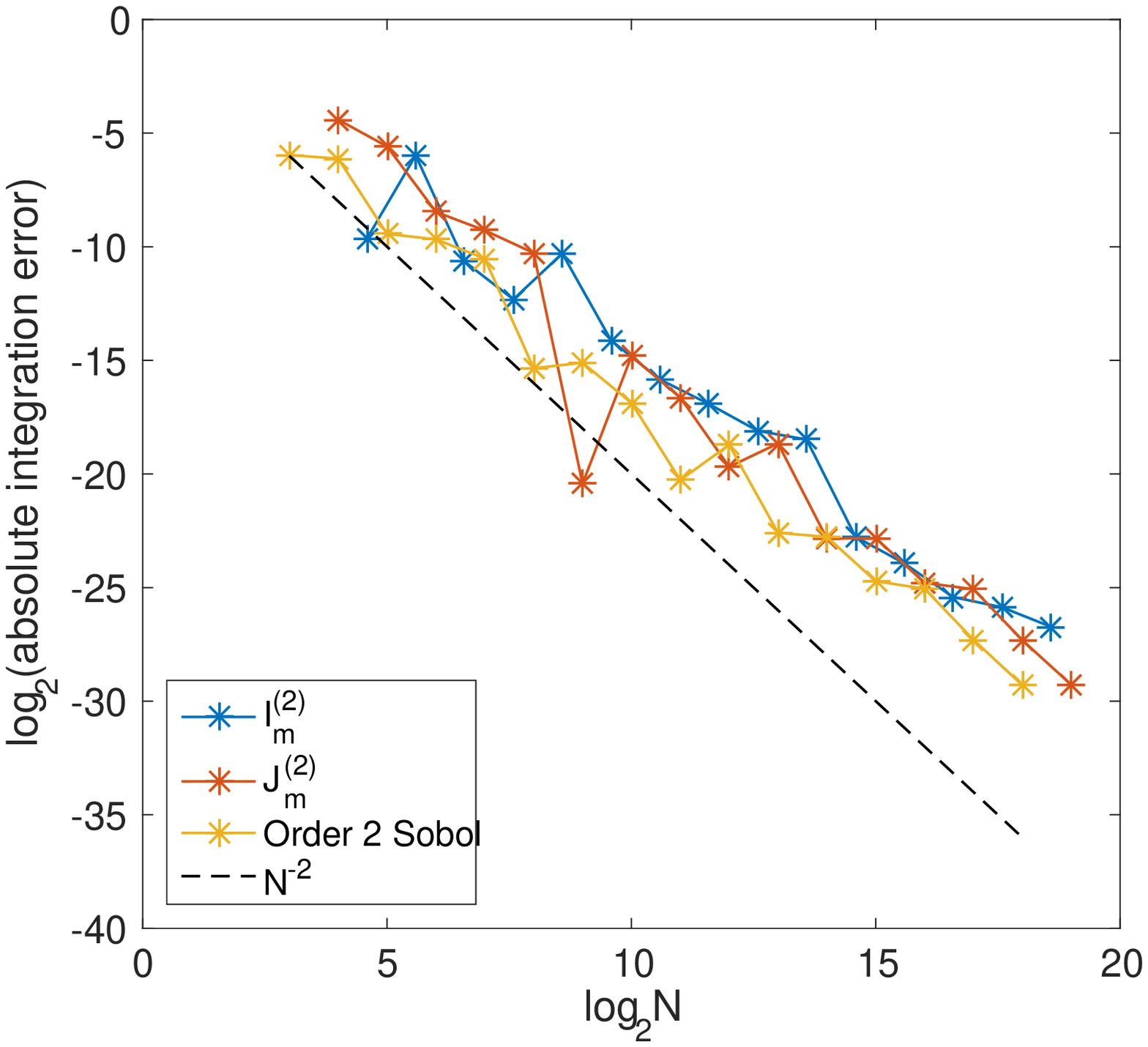}
\caption{Comparison of integration errors by three algorithms for $f_3$ with $c_1=1.3$ (left) and $f_4$ with $c_2=1$ (bottom).}
\label{fig:test3-5}
\end{figure}

\section*{Acknowledgements}
The author would like to thank Tomohiko Hironaka and Takehito Yoshiki for useful discussions and comments. The comments and suggestions made by the anonymous referees improving the exposition of this paper are greatly appreciated.

\end{document}